\RequirePackage[l2tabu,orthodox]{nag}		
\documentclass[reqno]{amsart}
\usepackage[margin=1.5in,bottom=1.25in,top=1.25in]{geometry}		

\usepackage{amsmath}		
\usepackage{amssymb}		
\usepackage{amsfonts}		
\usepackage{amsthm}		
\usepackage[foot]{amsaddr}		

\usepackage{mathtools}		
\mathtoolsset{%
}

\usepackage[utf8]{inputenc}		
\usepackage[T1]{fontenc}		

\usepackage{dsfont}		

\usepackage[scaled]{inconsolata}

\usepackage[proportional,tabular,lining,sf,mono=false]{libertine}

\usepackage[
cal=cm,
]
{mathalfa}

\usepackage[labelfont={bf,small},labelsep=colon,font=small]{caption}	
\usepackage[dvipsnames,svgnames]{xcolor}		
\colorlet{MyRed}{Crimson!80!Black}
\colorlet{MyGreen}{DarkGreen!85!Black}
\colorlet{MyBlue}{MediumBlue!80!Black}

\newcommand{\afterhead}{.}		

\newcommand{\para}[1]{\medskip\paragraph{\textbf{#1\afterhead}}}

\usepackage{pifont}		
\newcommand{\cmark}{\text{\color{MyGreen}\ding{51}}}		
\newcommand{\xmark}{\text{\color{MyRed}\ding{55}}}

\usepackage{subcaption}		
\usepackage{tikz}		
\usetikzlibrary{calc,patterns}

\usepackage{array}		
\usepackage{booktabs}		
\usepackage[inline,shortlabels]{enumitem}		
\setenumerate{itemsep=\smallskipamount,topsep=\medskipamount}

\usepackage{microtype}		

\usepackage{acronym}		
\usepackage{latexsym}		
\usepackage{xspace}		

\usepackage[numbers,sort&compress]{natbib}		

\bibpunct[, ]{[}{]}{,}{n}{,}{,}

\newcommand{\citefull}[2][]{\citeauthor{#2} \cite[#1]{#2}}

\usepackage{hyperref}
\hypersetup{
colorlinks=true,
linktocpage=true,
pdfstartview=FitH,
breaklinks=true,
pdfpagemode=UseNone,
pageanchor=true,
pdfpagemode=UseOutlines,
plainpages=false,
bookmarksnumbered,
bookmarksopen=false,
bookmarksopenlevel=1,
hypertexnames=true,
pdfhighlight=/O,
urlcolor=MyRed,linkcolor=MyBlue,citecolor=MyBlue,	
pdftitle={},
pdfauthor={},
pdfsubject={},
pdfkeywords={},
pdfcreator={pdfLaTeX},
pdfproducer={LaTeX with hyperref}
}

\usepackage[sort&compress,capitalize,nameinlink]{cleveref}		
\crefname{assumption}{Assumption}{Assumptions}

\usepackage[textwidth=30mm,disable]{todonotes}		

\newcommand{\debug}[1]{#1}		
\newcommand{\revise}[1]{#1}		

\def\beginrev{\color{black}}		
\def\endedit{\color{black}}		

\usepackage{algorithm}		
\usepackage{algpseudocode}		

\theoremstyle{plain}
\newtheorem{theorem}{Theorem}		
\newtheorem{lemma}{Lemma}		
\newtheorem{proposition}{Proposition}		

\newtheorem*{corollary*}{Corollary}		

\theoremstyle{definition}
\newtheorem{definition}{Definition}		
\newtheorem{example}{Example}		

\newtheorem*{definition*}{Definition}		
\newtheorem*{assumption*}{Assumptions}		

\theoremstyle{remark}
\newtheorem{remark}{Remark}		

\newtheorem*{remark*}{Remark}		
\newtheorem*{example*}{Example}		

\def\endenv{\hfill{\small$\blacktriangleleft$}}		

\newcounter{proofpart}

\numberwithin{example}{section}		

\newcommand{\newmacro}[2]{\newcommand{#1}{\debug{#2}}}		
\newcommand{\newop}[2]{\DeclareMathOperator{#1}{\debug{#2}}}		

\DeclarePairedDelimiter{\bracks}{[}{]}		
\DeclarePairedDelimiter{\parens}{(}{)}		

\DeclarePairedDelimiter{\abs}{\lvert}{\rvert}		

\DeclarePairedDelimiterX{\setdef}[2]{\{}{\}}{#1:#2}		
\DeclarePairedDelimiterXPP{\exclude}[1]{\mathopen{}\setminus}{\{}{\}}{}{#1}

\newcommand{\eg}{e.g.,\xspace}		
\newcommand{\ie}{i.e.,\xspace}		

\newcommand{\txs}{\textstyle}		

\newcommand{\alt}[1]{#1'}		

\newcommand{\N}{\mathbb{N}}		
\newcommand{\R}{\mathbb{R}}		

\DeclareMathOperator{\bigoh}{\mathcal O}		

\newcommand{\eps}{\varepsilon}		

\newcommand{\insum}{\sum\nolimits}		

\newmacro{\vdim}{d}		
\newmacro{\vecspace}{\R^{\vdim}}		
\newmacro{\subspace}{\mathcal{W}}		

\newmacro{\vvec}{v}		
\newmacro{\bvec}{e}		
\newmacro{\unitvec}{z}		

\DeclarePairedDelimiterX{\braket}[2]{\langle}{\rangle}{#1,#2}		

\newmacro{\dspace}{\R^{\vdim}}		
\newmacro{\dvec}{v}		
\newmacro{\dbvec}{\eps}		

\newmacro{\ones}{\mathbf{1}}		
\newmacro{\mat}{A}		

\DeclareMathOperator{\relint}{ri}		

\newmacro{\cvx}{\mathcal{C}}		
\newmacro{\tvec}{z}		
\newmacro{\tanspace}{\mathcal{Z}}		
\newmacro{\subd}{\partial}		

\DeclareMathOperator*{\argmin}{arg\,min}		
\DeclareMathOperator{\gap}{Gap}         
\newop{\dgap}{Gap^{\ast}}         

\newmacro{\actions}{\mathcal{K}}
\newmacro{\points}{\mathcal{X}}		
\newmacro{\point}{x}		
\newmacro{\pointalt}{\alt\point}		

\newmacro{\dpoints}{\mathcal{Y}}		
\newmacro{\dpoint}{y}		
\newmacro{\dpointalt}{\alt\dpoint}		
\newmacro{\obj}{f}		
\newmacro{\sobj}{F}		
\newmacro{\oper}{V}		

\newmacro{\gvec}{V}		
\newmacro{\gbound}{M}		

\newmacro{\vecfield}{V}		
\newmacro{\vbound}{M}		

\newmacro{\str}{\ell}		
\newmacro{\smooth}{L}		

\newcommand{\sol}[1][\point]{#1^{\ast}}		
\newcommand{\sols}{\sol[\points]}		

\newcommand{\test}[1][\point]{\hat#1}		

\newmacro{\subpoints}{\points^{\circ}}

\newmacro{\minmax}{\mathcal{L}}		
\newmacro{\minvar}{\theta}		
\newmacro{\maxvar}{\phi}		
\newmacro{\minvars}{\Theta}		
\newmacro{\maxvars}{\Phi}		
\newmacro{\minvaralt}{\alt\minvar}
\newmacro{\maxvaralt}{\alt\maxvar}
\newmacro{\minsol}{\sol[\minvar]}		
\newmacro{\maxsol}{\sol[\maxvar]}		

\newmacro{\play}{i}		
\newmacro{\playalt}{j}		
\newmacro{\nPlayers}{N}		
\newmacro{\players}{\mathcal{\nPlayers}}		

\newmacro{\pure}{\alpha}		
\newmacro{\purealt}{\beta}		
\newmacro{\nPures}{A}		
\newmacro{\pures}{\mathcal{\nPures}}		

\newmacro{\strats}{\points}		
\newmacro{\intstrats}{\strats^{\circ}}		

\newmacro{\cost}{c}		
\newmacro{\loss}{\ell}		
\newmacro{\pay}{u}		
\newmacro{\payv}{v}		
\newmacro{\pot}{\obj}		

\newmacro{\eq}{\sol}		
\newmacro{\game}{\mathcal{G}}		
\newmacro{\fingame}{\Gamma}		

\DeclareMathOperator{\Eucl}{\Pi}		

\newmacro{\hreg}{h}		
\newmacro{\hconj}{\hreg^{\ast}}		
\newmacro{\breg}{D}		
\newmacro{\prox}{P}		
\newmacro{\mirror}{Q}		
\newmacro{\fench}{F}		
\newmacro{\hstr}{K}		
\newmacro{\depth}{H}		
\newmacro{\radius}{R}		

\newmacro{\source}{O}
\newmacro{\sources}{\mathcal{O}}
\newmacro{\sink}{D}
\newmacro{\sinks}{\mathcal{D}}

\newmacro{\pair}{\play}
\newmacro{\pairalt}{\playalt}
\newmacro{\nPairs}{\nPlayers}
\newmacro{\pairs}{\players}

\newmacro{\rate}{m}
\newmacro{\totrate}{M}
\newmacro{\relrate}{\mu}

\newmacro{\flow}{f}
\newmacro{\flows}{\mathcal{F}}

\newmacro{\load}{\point}
\newmacro{\loads}{\points}

\newmacro{\nResources}{N}
\newmacro{\resources}{\mathcal{\nResources}}
\newmacro{\resource}{i}
\newmacro{\resourcealt}{j}

\newmacro{\capa}{\mu}
\newmacro{\inflow}{\rho}
\newmacro{\late}{\tau}

\DeclareMathOperator{\ex}{\mathbb{E}}		
\DeclareMathOperator{\prob}{\mathbb{P}}		

\newmacro{\sample}{\omega}		
\newmacro{\samples}{\Omega}		
\newmacro{\filter}{\mathcal{F}}		

\newmacro{\dkl}{\mathrm{KL}}		

\DeclarePairedDelimiterXPP{\exof}[1]{\ex}{[}{]}{}{
 #1}

\DeclarePairedDelimiterXPP{\probof}[1]{\prob}{(}{)}{}{
 #1}

\newmacro{\orcl}{\mathsf{V}}		
\newmacro{\signal}{V}		
\newmacro{\sdiff}{\delta}		

\newmacro{\bias}{b}		
\newmacro{\noise}{U}		
\newmacro{\snoise}{\xi}		
\newmacro{\noisedev}{\sigma}		
\newmacro{\noisevar}{\noisedev^{2}}		

\newmacro{\prestart}{0}		
\newmacro{\halfstart}{\debug{1/2}}		
\newmacro{\start}{1}		
\newmacro{\running}{1,2,\dotsc}		
\newmacro{\halfrunning}{1,3/2,\dotsc}		
\newmacro{\run}{t}		
\newmacro{\runalt}{s}		
\newmacro{\nRuns}{T}		

\newcommand{\avg}[1][\state]{\bar{#1}}		
\newcommand{\new}[1]{#1^{+}}		

\newmacro{\state}{X}		
\newmacro{\score}{y}		
\newmacro{\invec}{v}		

\newmacro{\step}{\gamma}		
\newmacro{\temp}{\eta}		
\newmacro{\shrink}{\theta}		

\newcommand{\init}[1][\state]{\debug{#1}_{\start}}		
\newcommand{\iter}[1][\state]{\debug{#1}_{\runalt}}		
\newcommand{\curr}[1][\state]{\debug{#1}_{\run}}		
\newcommand{\lead}[1][\state]{\debug{#1}_{\run+1/2}}		
\renewcommand{\next}[1][\state]{\debug{#1}_{\run+1}}		

\DeclareMathOperator{\dom}{dom}		
\DeclareMathOperator{\one}{\mathds{1}}		

\newmacro{\set}{\mathcal{S}}		
\newcommand{\from}{\colon}		

\newcommand{\defeq}{\coloneqq}		

\newmacro{\base}{p}		
\newmacro{\basealt}{q}		

\newmacro{\graph}{\mathcal{G}}
\newmacro{\vertices}{\mathcal{V}}
\newmacro{\edges}{\mathcal{E}}

\DeclareMathOperator{\cl}{cl}		
\DeclareMathOperator{\dist}{dist}		

\newmacro{\cpt}{\mathcal{K}}						
\newmacro{\nhd}{U}		
\newmacro{\open}{U}		

\DeclarePairedDelimiterX{\product}[2]{\langle}{\rangle}{#1,#2}		
\DeclarePairedDelimiter{\norm}{\lVert}{\rVert}		
\DeclarePairedDelimiterXPP{\dnorm}[1]{}{\lVert}{\rVert}{_{\ast}}{#1}		

\newmacro{\gmat}{g}		
\newmacro{\dfun}{\dist_{\gmat}}
\newmacro{\ball}{\mathbb{B}}		
\newmacro{\sphere}{\mathbb{S}}		
\newmacro{\fins}{F}       

\newcommand{\method}{\debug{AdaProx}\xspace}
\newcommand{\adagrad}{\debug{AdaGrad}\xspace}

\newmacro{\bregdim}{D_{0}}
\newmacro{\res}{Z}
\newmacro{\aux}{\tilde\res}		
\newmacro{\auxalt}{W}		
\newmacro{\gmap}{\mathsf{G}}

\newcommand{\PM}[1]{\todo[color=DodgerBlue!20!LightGray,author=\textbf{PM}]{\small #1\\}}
\newmacro{\Shah}{Shahshahani\xspace}

\begin{document}


\title
[Adaptive Extra-Gradient Methods for Min-Max Optimization and Games]
{Adaptive Extra-Gradient Methods\\for Min-Max Optimization and Games}		

\author[K.~Antonakopoulos]{Kimon Antonakopoulos$^{\ast,\lowercase{c}}$}
\address{$^{\ast}$ Inria, Univ. Grenoble Alpes, CNRS, Grenoble INP, LIG 38000 Grenoble, France.}
\address{$^{c}$ Corresponding author.}
\email{kimon.antonakopoulos@inria.fr}

\author[E.~V.~Belmega]{E.~Veronica Belmega$^{\ddag}$}
\address{$^{\ddag}$ ETIS UMR8051, CY University, ENSEA, CNRS, F-95000, Cergy, France}
\email{belmega@ensea.fr}

\author[P.~Mertikopoulos]{Panayotis Mertikopoulos$^{\ast,\diamond}$}
\address{$^{\diamond}$ Criteo AI Lab.}
\email{panayotis.mertikopoulos@imag.fr}


\newcommand{\acli}[1]{\textit{\acl{#1}}}		
\newcommand{\aclip}[1]{\textit{\aclp{#1}}}		
\newcommand{\acdef}[1]{\textit{\acl{#1}} \textup{(\acs{#1})}\acused{#1}}		
\newcommand{\acdefp}[1]{\emph{\aclp{#1}} \textup(\acsp{#1}\textup)\acused{#1}}	

\newacro{MD}{mirror descent}
\newacro{OMD}{optimistic mirror descent}
\newacro{MP}{mirror-prox}
\newacro{GMP}{generalized mirror-prox}
\newacro{AMP}{adaptive mirror-prox}
\newacro{EG}{extra-gradient}
\newacro{GRAAL}{golden ratio algorithm}
\newacro{BL}{Bach\textendash Levy}
\newacro{FB}{forward-backward}

\newacro{FCFS}{first-come, first-served}
\newacro{VI}{variational inequality}
\newacroplural{VI}{variational inequalities}
\newacro{SVI}{Stampacchia variational inequality}
\newacro{SP}{saddle-point}

\newacro{KL}{Kullback\textendash Leibler}
\newacro{KLow}[KL]{Kurdyka\textendash \L ojasiewicz}
\newacro{ACGAN}{auxiliary classifier GAN}
\newacro{GAN}{generative adversarial network}
\newacro{WGAN}{Wasserstein GAN}

\newacro{LHS}{left-hand side}
\newacro{RHS}{right-hand side}
\newacro{iid}[i.i.d.]{independent and identically distributed}
\newacro{lsc}[l.s.c.]{lower semi-continuous}
\newacro{NE}{Nash equilibrium}
\newacroplural{NE}[NE]{Nash equilibria}
\newacro{WE}{Wardrop equilibrium}
\newacroplural{WE}[WE]{Wardrop equilibria}

\begin{abstract}

We present a new family of min-max optimization algorithms that automatically exploits the geometry of the gradient data observed at earlier iterations to perform more informative extra-gradient steps in later ones.
Thanks to this adaptation mechanism, the proposed methods automatically detect
\revise{whether the problem is smooth or not,}
without requiring any prior tuning by the optimizer.
As a result, the algorithm simultaneously achieves \revise{order-optimal} convergence rates, \ie it  converges to an $\eps$-optimal solution within $\bigoh(1/\eps)$ iterations in smooth problems, and within $\bigoh(1/\eps^2)$ iterations in non-smooth ones.
Importantly, these guarantees do not require any of the standard boundedness or Lipschitz continuity conditions that are typically assumed in the literature;
in particular, they apply even to problems with \revise{singularities} (such as resource allocation problems and the like).
This adaptation is achieved through the use of a geometric apparatus based on Finsler metrics and a suitably chosen mirror-prox template that allows us to derive sharp convergence rates for the methods at hand.
\end{abstract}

\allowdisplaybreaks		
\acresetall		
\maketitle

\section{Introduction}
\label{sec:introduction}

The surge of recent breakthroughs in \acp{GAN} \citep{GPAM+14}, robust reinforcement learning \cite{PDSG17}, and other adversarial learning models \cite{MMST+18} has sparked renewed interest in the theory of min-max optimization problems and games.
In this broad setting, it has become empirically clear that, ceteris paribus, the simultaneous training of two (or more) antagonistic models faces drastically new challenges relative to the training of a single one.
Perhaps the most prominent of these challenges is the appearance of cycles and recurrent (or even chaotic) behavior in min-max games.
This has been studied extensively in the context of learning in bilinear games,
\revise{in both continuous \cite{PS14,MPP18,FVGP19a} and discrete time \cite{DISZ18,MLZF+19,GHPL+19,GBVV+19},}
and the methods proposed to overcome recurrence typically focus on mitigating the rotational component of min-max games.

The method with the richest history in this context is the \acdef{EG} algorithm of \citefull{Kor76} and its variants.
\revise{The \ac{EG} algorithm} exploits the Lipschitz smoothness of the problem and, if coupled with \revise{a Polyak\textendash Ruppert averaging scheme,}
it achieves an $\bigoh(1/\nRuns)$ rate of convergence in smooth, convex-concave min-max problems \citep{Nem04}.
This rate is known to be tight \citep{Nem92,OX19} but, in order to achieve it, the original method requires the problem's Lipschitz constant to be known in advance.
If the problem is not Lipschitz smooth (or the algorithm is run with a vanishing step-size schedule), the method's rate of convergence drops to $\bigoh(1/\sqrt{\nRuns})$.

\para{Our contributions}

Our aim in this paper is to provide an algorithm that automatically adapts to smooth / non-smooth min-max problems and games,
and
achieves \revise{order-optimal} rates in both classes
without requiring any prior tuning by the optimizer.
In this regard, we propose a flexible algorithmic scheme, which we call \method, and which exploits gradient data observed at earlier iterations to perform more informative extra-gradient steps in later ones.
Thanks to this mechanism, and to the best of our knowledge, \method is the first algorithm that simultaneously achieves the following:
\begin{enumerate}[leftmargin=1cm]
\item
An $\bigoh\parens[\big]{1/\sqrt{\nRuns}}$ convergence rate in non-smooth problems and $\bigoh(1/\nRuns)$ in smooth ones.
\item
Applicability to min-max problems and games where the standard boundedness / Lipschitz continuity conditions required in the literature do not hold.
\item
\revise{Convergence without prior knowledge of the problem's parameters (\eg whether the problem's defining vector field is smooth or not, its smoothness modulus if it is, etc.).}
\end{enumerate}

Our proposed method achieves the above by fusing the following ingredients:
\begin{enumerate*}
[\itshape a\upshape)]
\item
a family of local norms \textendash\ a \emph{Finsler metric} \textendash\ capturing any singularities in the problem at hand;
\item
a suitable \acl{MP} template;
and
\item
an adaptive step-size policy in the spirit of \citefull{RS13-NIPS}.
\end{enumerate*}
\revise{We also show that, under a suitable coherence assumption,} the sequence of iterates generated by the algorithm converges, thus providing an appealing alternative to iterate averaging in cases where the method's ``last iterate'' is more appropriate
\revise{(for instance, if using \method to solve non-monotone problems).}

\para{Related work}

There have been several works improving on the guarantees of the original \acl{EG}/\acl{MP} template.
We review the most relevant of these works below;
for convenience, we also tabulate these contributions in \cref{tab:related}.
Because many of these works appear in the literature on \aclp{VI} \cite{FP03}, we also use this language in the sequel.

In unconstrained problems with an operator that is locally Lipschitz continuous (but not necessarily globally so), the \acdef{GRAAL} \cite{Mal19} achieves convergence without requiring prior knowledge of the problem's Lipschitz parameter.
However, \ac{GRAAL} provides no rate guarantees for non-smooth problems \textendash\ and hence, a fortiori, no interpolation guarantees either.
By contrast, such guarantees are provided in problems with a bounded domain by the \acdef{GMP} algorithm of \citefull{SGDA+18} under the umbrella of Hölder continuity.
Still, nothing is known about the convergence of \ac{GRAAL}\,/\,\ac{GMP} in problems \revise{with singularities (\ie when the problem's defining vector field blows up at a boundary point of the problem's domain).}

Another method that simultaneously achieves an $\bigoh(1/\sqrt{\nRuns})$ rate in non-smooth problems and an $\bigoh(1/\nRuns)$ rate in smooth ones is the recent algorithm of \citefull{BL19}.\acused{BL}
The \ac{BL} algorithm employs an adaptive, \adagrad-like step-size policy which allows the method to interpolate between the two regimes \textendash\ and this, even with noisy gradient feedback.
On the negative side, the \ac{BL} algorithm requires a bounded domain with a (Bregman) diameter that is known in advance;
as a result, its theoretical guarantees do not apply to problems with an unbounded domain.
In addition, the \ac{BL} algorithm makes crucial use of operator boundedness and Lipschitz continuity;
extending the \ac{BL} method beyond this standard framework is a highly non-trivial endeavor which formed a big part of this paper's motivation.

\revise{Operators with singularities} were treated in a recent series of papers \cite{ABM19,GDST19,SGTP+19} by means of a ``Bregman continuity'' or ``Lipschitz-like'' condition in the spirit of \citefull{BBT17} and \citefull{LFN18}.
Albeit different, the adaptive methods presented in \cite{ABM19,GDST19} are both \revise{order-optimal} in the smooth case, without requiring any knowledge of the problem's smoothness modulus.
On the other hand, like \ac{GRAAL} \textendash\ but unlike \ac{GMP} \textendash\ 
\revise{they do not provide any rate interpolation guarantees between smooth and non-smooth problems.}
Finally, the method of \cite{SGTP+19} provides an ``inexact model'' framework that unifies the approach of \cite{GDST19} and \cite{SGDA+18}, providing rate interpolation in the Hölder case and convergence in problems with singularities;%
\footnote{Personal communication with P.~Dvurechensky suggests that the method of \cite{SGTP+19} can be further adapted to problems with singularities under the metric boundedness framework presented in this paper.}
however, in problems with an unbounded domain, it still requires an initial guess of a compact set containing a solution.


\begin{table}[tbp]
\centering
\footnotesize
\renewcommand{\arraystretch}{1.2}

\begin{tabular}{lcccccc}
	&\acs{EG} \citep{Kor76,Nem04}
	&\acs{GRAAL} \citep{Mal19}
	&\acs{GMP} \citep{SGDA+18}
	&\acs{AMP} \citep{ABM19,GDST19}
	&\acs{BL} \citep{BL19}
	&\method [ours]
	\\
\hline
Param.~Agnostic
	&\xmark
	&\cmark
	&Partial
	&\cmark
	&Partial
	&\cmark
	\\
\hline
Rate Interpolation
	&\xmark
	&\xmark
	&\cmark
	&\xmark
	&\cmark
	&\cmark
	\\
\hline
Unb.~Domain
	&\xmark
	&\cmark
	&\xmark
	&\xmark
	&\xmark
	&\cmark
	\\
\hline
Singularities
	&\xmark
	&\xmark
	&\xmark
	&\cmark
	&\xmark
	&\cmark
	\\
\hline
\end{tabular}
\smallskip
\caption{%
Overview of related work.
For the purposes of this table,
``parameter-agnostic'' means that the method does not require prior knowledge of the parameters of the problem it was designed to solve (Lipschitz modulus, domain diameter, etc.);
``rate interpolation'' means that the algorithm's convergence rate is $\bigoh(1/\nRuns)$ or $\bigoh\parens[\big]{1/\sqrt{\nRuns}}$ in smooth\,/\,non-smooth problems respectively;
\beginrev
``unbounded domain'' is self-explanatory;
and, finally, ``singularities'' means that the problem's defining vector field may blow up at a boundary point of the problem's domain.
\endedit}
\label{tab:related}
\end{table}


\section{Problem Setup and Blanket Assumptions}
\label{sec:setup}

We begin in this section by reviewing some basics for min-max problems and games.

\subsection{Min-max / Saddle-point problems}
\label{sec:minmax}

A \emph{min-max game} is a saddle-point problem of the form
\begin{equation}
\label{eq:minmax}
\tag{SP}
\min_{\minvar\in\minvars} \max_{\maxvar\in\maxvars} \minmax(\minvar,\maxvar)
\end{equation}
where $\minvars$, $\maxvars$ are convex subsets of some ambient real space and $\minmax\from\minvars \times \maxvars \to \R$ is the problem's \emph{loss function}.
In the game-theoretic interpretation of \eqref{eq:minmax},
the player controlling $\minvar$ seeks to minimize $\minmax(\minvar,\maxvar)$ for any value of the maximization variable $\maxvar$,
while the player controlling $\maxvar$ seeks to maximize $\minmax(\minvar,\maxvar)$ for any value of the minimization variable $\minvar$.
Accordingly, solving \eqref{eq:minmax} consists of finding a \acdef{NE}, \ie an action profile $(\minsol,\maxsol) \in \minvars \times \maxvars$ such that
\begin{equation}
\label{eq:NE-sol}
\minmax(\minsol,\maxvar)
	\leq \minmax(\minsol,\maxsol)
	\leq \minmax(\minvar,\maxsol)
	\quad
	\text{for all $\minvar\in\minvars$, $\maxvar\in\maxvars$}.
\end{equation}
By the minimax theorem of \citefull{vN28}, \aclp{NE} are guaranteed to exist when
$\minvars,\maxvars$ are compact and $\minmax$ is convex-concave (\ie convex in $\minvar$ and concave in $\maxvar$).
Much of our paper is motivated by the question of calculating a \acl{NE} $(\minsol,\maxsol)$ of \eqref{eq:minmax} in the context of von Neumann's theorem;
we expand on this below.

\subsection{Games}
\label{sec:NE}

Going beyond the min-max setting, a \emph{continuous game in normal form} is defined as follows:
First, consider a finite set of players $\players = \{1,\dotsc,\nPlayers\}$, each with their own action space $\actions_{\play} \in \R^{\vdim_{\play}}$ (assumed convex \revise{but possibly not closed}).
During play, each player selects an action $\point_{\play}$ from $\actions_{\play}$ with the aim of minimizing a loss determined by the ensemble $\point \defeq (\point_{\play};\point_{-\play}) \defeq (\point_{1},\dotsc,\point_{\nPlayers})$ of all players' actions.
In more detail, writing $\actions \defeq \prod_{\play}\actions_{\play}$ for the game's total action space, we assume that the loss incurred by the $\play$-th player is $\loss_{\play}(\point_{\play};\point_{-\play})$, where $\loss_{\play}\from\actions\to\R$ is the player's \emph{loss function}.

In this context, a \acl{NE} is any action profile $\sol \in \actions$ that is \emph{unilaterally stable}, \ie
\begin{equation}
\label{eq:NE}
\tag{NE}
\loss_{\play}(\sol_{\play};\sol_{-\play})
	\leq \loss_{\play}(\point_{\play};\sol_{-\play})
	\quad
	\text{for all $\point_{\play}\in\actions_{\play}$ and all $\play\in\players$}.
\end{equation}
If each $\actions_{\play}$ is compact and $\loss_{\play}$ is convex in $\point_{\play}$, existence of \aclp{NE} is guaranteed by the theorem of \citefull{Deb52}.
Given that a min-max problem can be seen as a two-player zero-sum game with $\loss_{1} = \minmax$, $\loss_{2} = -\minmax$, von Neumann's theorem may in turn be seen as a special case of Debreu's;
in the sequel, we describe a first-order characterization of \aclp{NE} that encapsulates both.

In most cases of interest, the players' loss functions are \emph{individually subdifferentiable} on a subset $\points$ of $\actions$ with $\relint\actions \subseteq \points \subseteq \actions$ \citep{Roc70,HUL01}.
This means that there exists a (possibly discontinuous) vector field $\vecfield_{\play}\from\points\to\R^{\vdim_{\play}}$ such that
\begin{equation}
\label{eq:subfield}
\loss_{\play}(\pointalt_{\play};\point_{-\play})
	\geq \loss_{\play}(\point_{\play};\point_{-\play})
	+ \braket{\vecfield_{\play}(\point)}{\pointalt_{\play} - \point_{\play}}
\end{equation}
for all $\point\in\points$, $\pointalt\in\actions$ \revise{and all $\play\in\players$} \citep{HUL01}.
In the simplest case, if $\loss_{\play}$ is differentiable at $\point$, then $\vecfield_{\play}(\point)$ can be interpreted as the gradient of $\loss_{\play}$ with respect to $\point_{\play}$.
The \emph{raison d'être} of the more general definition \eqref{eq:subfield} is that it allows us to treat non-smooth loss functions that are common in machine learning (such as $L^{1}$-regularized losses).
We make this distinction precise below:
\begin{enumerate}[topsep=0pt]
\item
If there is no continuous vector field $\vecfield_{\play}(\point)$ satisfying \eqref{eq:subfield}, the game is called \emph{non-smooth}.
\item
If there is a continuous vector field $\vecfield_{\play}(\point)$ satisfying \eqref{eq:subfield}, the game is called \emph{smooth}.
\end{enumerate}
\smallskip

\begin{remark*}
We stress here that the adjective ``smooth'' refers to the game itself:
for instance, if $\loss(\point) = \abs{\point}$ for $\point\in\R$, the game is not smooth and any $\vecfield$ satisfying \eqref{eq:subfield} is discontinuous at $0$.
In this regard, the above boils down to whether the (individual) subdifferential of each $\loss_{\play}$ admits a continuous selection.
\end{remark*}

\beginrev
\subsection{Resource allocation and equilibrium problems}
\label{sec:WE}

The notion of a \acl{NE} captures the unilateral minimization of the players' individual loss functions.
In many pratical cases of interest, a notion of equilibrium is still relevant, even though it is not necessarily attached to the minimization of individual loss functions.
Such problems are known as ``equilibrium problems'' \cite{FP03,LRS19};
to avoid unnecessary generalities, we focus here on a relevant problem that arises in distributed computing architectures (such as GPU clusters and the like).

To state the problem, consider a distributed computing grid consisting of $\nResources$ parallel processors that serve demands arriving at a rate of $\inflow$ per unit of time (measured \eg in flop/s).
If the maximum processing rate of the $\resource$-th node is $\capa_{\resource}$ (without overclocking), and jobs are buffered and served on a \ac{FCFS} basis, the mean time required to process a unit demand at the $\resource$-th node is given by the Kleinrock M/M/1 response function $\late_{\resource}(\load_{\resource}) = 1/(\capa_{\resource} - \load_{\resource})$, where $\load_{\resource}$ denotes the node's \emph{load} \citep{BG92}.
Accordingly, the set of feasible loads that can be processed by the grid is $\loads \defeq \setdef{(\load_{1},\dotsc,\load_{\nResources})}{0 \leq \load_{\resource} < \capa_{\resource}, \load_{1} + \dotsm + \load_{\nResources} = \inflow}$.

In this context, a load profile $\sol[\load] \in \loads$ is said to be \emph{balanced} if no infinitesimal process can be better served by buffering it at a different node \citep{NRTV07};
formally, this amounts to the so-called \acli{WE} condition
\begin{equation}
\label{eq:WE}
\tag{WE}
\late_{\resource}(\sol[\load_{\resource}])
	\leq \late_{\resourcealt}(\sol[\load_{\resourcealt}])
	\quad
	\text{for all $\resource,\resourcealt\in\resources$ with $\sol[\load]_{\resource} > 0$}.
\end{equation}
We note here a crucial difference between \eqref{eq:WE} and \eqref{eq:NE}:
if we view the grid's computing nodes as ``players'', the constraint $\sum_{\resource}\load_{\resource} = \inflow$ means that there is no allowable unilateral deviation $(\sol[\load_{\resource}];\sol[\load_{-\resource}]) \mapsto (\load_{\resource};\sol[\load_{-\resource}])$ with $\load_{\resource} \neq \sol[\load_{\resource}]$.
As a result, \eqref{eq:NE} is meaningless as a requirement for this equilibrium problem.

As we discuss below, this resource allocation problem will require the full capacity of our framework.


\subsection{Variational inequalities}
\label{sec:VI}

Importantly, all of the above problems can be restated as a \acli{VI} of the form
\begin{equation}
\label{eq:VI}
\tag{VI}
\text{Find $\sol\in\points$ such that}
	\;\;
	\braket{\vecfield(\sol)}{\point - \sol}
	\geq 0
	\;\;
	\text{for all $\point\in\points$}.
\end{equation}
In the above, $\points$ is a convex subset of $\vecspace$ (not necessarily closed) that represents the problem's \emph{domain}.
The problem's \emph{defining vector field} $\vecfield\from\points\to\vecspace$ is then given as follows:
In min-max problems and games, $\vecfield$ is any field satisfying \eqref{eq:subfield};
otherwise, in equilibrium problems of the form \eqref{eq:WE}, the components of $\vecfield$ are $\vecfield_{i} = \late_{i}$ (we leave the details of this verification to the reader).
\endedit


This equivalent formulation is quite common in the literature on min-max / equilibrium problems \cite{FP03,FK07,MZ19,LRS19}, and it is often referred to as the ``vector field formulation'' \citep{BRMF+18,CGFLJ19,HIMM20}.
Its usefulness lies in that it allows us to abstract away from the underlying game-theoretic complications (multiple indices, individual subdifferentials, etc.) and provides a unifying framework for a wide range of problems in machine learning, signal processing, operations research, and many other fields \citep{FP03,SFPP10}.
For this reason, our analysis will focus almost exclusively on solving \eqref{eq:VI}, and we will treat $\vecfield$ and $\points \subseteq \vecspace$, $\vdim = \sum_{\play} \vdim_{\play}$, as the problem's primitive data.

\subsection{Merit functions and monotonicity}

A widely used assumption in the literature on equilibrium problems and \aclp{VI} is the \emph{monotonicity condition}
\begin{equation}
\label{eq:MC}
\tag{MC}
\braket{\vecfield(\point) - \vecfield(\pointalt)}{\point - \pointalt}
	\geq 0
	\quad
	\text{for all $\point,\pointalt \in \points$}.
\end{equation}
In single-player games, monotonicity is equivalent to convexity of the optimizer's loss function;
in min-max games, it is equivalent to $\minmax$ being convex-concave \citep{LRS19};
etc.
In the absence of monotonicity, approximating an equilibrium is PPAD-hard \citep{DGP09}, so we will state most of our results under \eqref{eq:MC}.

Now, to assess the quality of a candidate solution $\test\in\points$, we will employ the \emph{restricted merit function}
\begin{align}
\label{eq:gap}
\gap_{\cvx}(\test)
	&= \sup\nolimits_{\point\in\cvx}
		\braket{\vecfield(\point)}{\test - \point},
\end{align}
where the ``\emph{test domain}'' $\cvx$ is a nonempty convex subset of $\points$ \citep{Nes07,JNT11,FP03}.
The motivation for this is provided by the following proposition:

\begin{proposition}
\label{prop:gap}
Let $\cvx$ be a nonempty convex subset of $\points$.
Then:
\begin{enumerate*}
[\itshape a\upshape)]
\item
$\gap_{\cvx}(\test) \geq 0$ whenever $\test\in\cvx$;
and
\item
if $\gap_{\cvx}(\test) = 0$ and $\cvx$ contains a neighborhood of $\test$, then $\test$ is a solution of \eqref{eq:VI}.
\end{enumerate*}
\end{proposition}

\cref{prop:gap} generalizes an earlier characterization by \citet{Nes07} and justifies the use of $\gap_{\cvx}(\point)$ as a merit function for \eqref{eq:VI};
to streamline our presentation, we defer the proof to the paper's supplement.
Moreover, to avoid trivialities, we will also assume that the solution set $\sols$
of \eqref{eq:VI} is nonempty and we will reserve the notation $\sol$ for solutions of \eqref{eq:VI}.
Together with monotonicity, this will be our only blanket assumption.

\section{The Extra-Gradient Algorithm and its Limits}
\label{sec:extragrad}

Perhaps the most widely used solution method for games and \acp{VI} is the \acdef{EG} algorithm of \citefull{Kor76} and its variants \citep{Pop80,RS13-NIPS,Mal15}.
This algorithm has a rich history in optimization, and it has recently attracted considerable interest in the fields of machine learning and AI,
\revise{see \eg \cite{DISZ18,GBVV+19,MLZF+19,CGFLJ19,HIMM19,HIMM20,MOP19b} and references therein.}

In its simplest form,
\revise{for problems with closed domains,}
the algorithm proceeds recursively as
\begin{equation}
\label{eq:EG}
\tag{EG}
\lead
	= \Eucl(\curr - \curr[\step] \curr[\signal]),
	\qquad
\next
	= \Eucl(\curr - \curr[\step] \lead[\signal]),
\end{equation}
where
$\Eucl(\point) = \argmin_{\pointalt\in\points} \norm{\pointalt - \point}$ is the Euclidean projection on $\points$,
$\curr[\signal] \defeq \vecfield(\curr)$ for $\run = \halfrunning$,
and
$\curr[\step] > 0$, is the method's step-size.
Then, running \eqref{eq:EG} for $\nRuns$ iterations, the algorithm returns the ``ergodic average''
\begin{equation}
\label{eq:ergodic}
\avg_{\nRuns}
	= \frac{\sum_{\run=\start}^{\nRuns} \curr[\step] \lead}{\sum_{\run=\start}^{\nRuns} \curr[\step]}.
\end{equation}
In this setting, the main guarantees for \eqref{eq:EG} date back to \cite{Nem04} and can be summarized as follows:
\begin{enumerate}
\addtolength{\itemsep}{\smallskipamount}

\item
\emph{For non-smooth problems} (discontinuous $\vecfield$):
Assume $\vecfield$ is \emph{bounded}, \ie
there exists some $\gbound >0$ such that
\begin{equation}
\label{eq:bounded}
\tag{BD}
\norm{\vecfield(\point)}
	\leq \gbound
	\quad
	\text{for all $\point \in \points$}.
\end{equation}
Then, if \eqref{eq:EG} is run with a step-size of the form $\curr[\step] \propto 1/\sqrt{\run}$, we have
\begin{equation}
\label{eq:EG-nonsmooth}
\gap_{\cvx}(\avg_{\nRuns})
	= \bigoh\parens[\big]{1/\sqrt{\nRuns}}.
\end{equation}

\item
\emph{For smooth problems} (continuous $\vecfield$):
Assume $\vecfield$ is \emph{$\smooth$-Lipschitz continuous}, \ie
\begin{equation}
\label{eq:LC}
\tag{LC}
\norm{\vecfield(\point) - \vecfield(\pointalt)}
	\leq \smooth \norm{\point - \pointalt}
	\quad
	\text{for all $\point,\pointalt \in \points$}.
\end{equation}
Then, if \eqref{eq:EG} is run with a constant step-size $\step < 1/\smooth$, we have
\begin{equation}
\label{eq:EG-smooth}
\gap_{\cvx}(\avg_{\nRuns})
	= \bigoh(1/\nRuns).
\end{equation}
\end{enumerate}
\smallskip

\begin{remark*}
In the above, $\norm{\cdot}$ is tacitly assumed to be the standard Euclidean norm.
Non-Euclidean considerations will play a crucial role in the sequel, but they are not necessary for the moment.
\end{remark*}

Importantly, the distinction between smooth and non-smooth problems cannot be lifted:
the bounds \eqref{eq:EG-nonsmooth} and \eqref{eq:EG-smooth} are tight in their respective problem classes and they cannot be improved without further assumptions \citep{Nem92,OX19}.
Moreover, we should also note the following:
\begin{enumerate}
\item
The algorithm changes drastically from the non-smooth to the smooth case:
non-smoothness requires $\curr[\step] \propto 1/\sqrt{\run}$, but such a step-size cannot achieve a \revise{fast} $\bigoh(1/\nRuns)$ rate.
\item
If \eqref{eq:EG} is run with a constant step-size, $\smooth$ must be known in advance;
otherwise, running \eqref{eq:EG} with an ill-adapted step-size ($\step > 1/\smooth)$ could lead to non-convergence.
\end{enumerate}

We illustrate this failure of \eqref{eq:EG} in \cref{fig:failure}.
As we discussed in the introduction, our aim in the sequel will be to provide a single, \emph{adaptive} algorithm that simultaneously achieves the following:
\begin{enumerate*}
[\itshape a\upshape)]
\item
an \revise{order-optimal} $\bigoh\parens[\big]{1/\sqrt{\nRuns}}$ convergence rate in non-smooth problems and $\bigoh(1/\nRuns)$ in smooth ones;
\item
convergence in problems where the boundedness / Lipschitz continuity conditions \eqref{eq:bounded} / \eqref{eq:LC} no longer hold;
and
\item
achieves all this without prior knowledge of the problem's parameters.
\end{enumerate*}


\begin{figure}[t]
\centering
\includegraphics[width=0.3\textwidth]{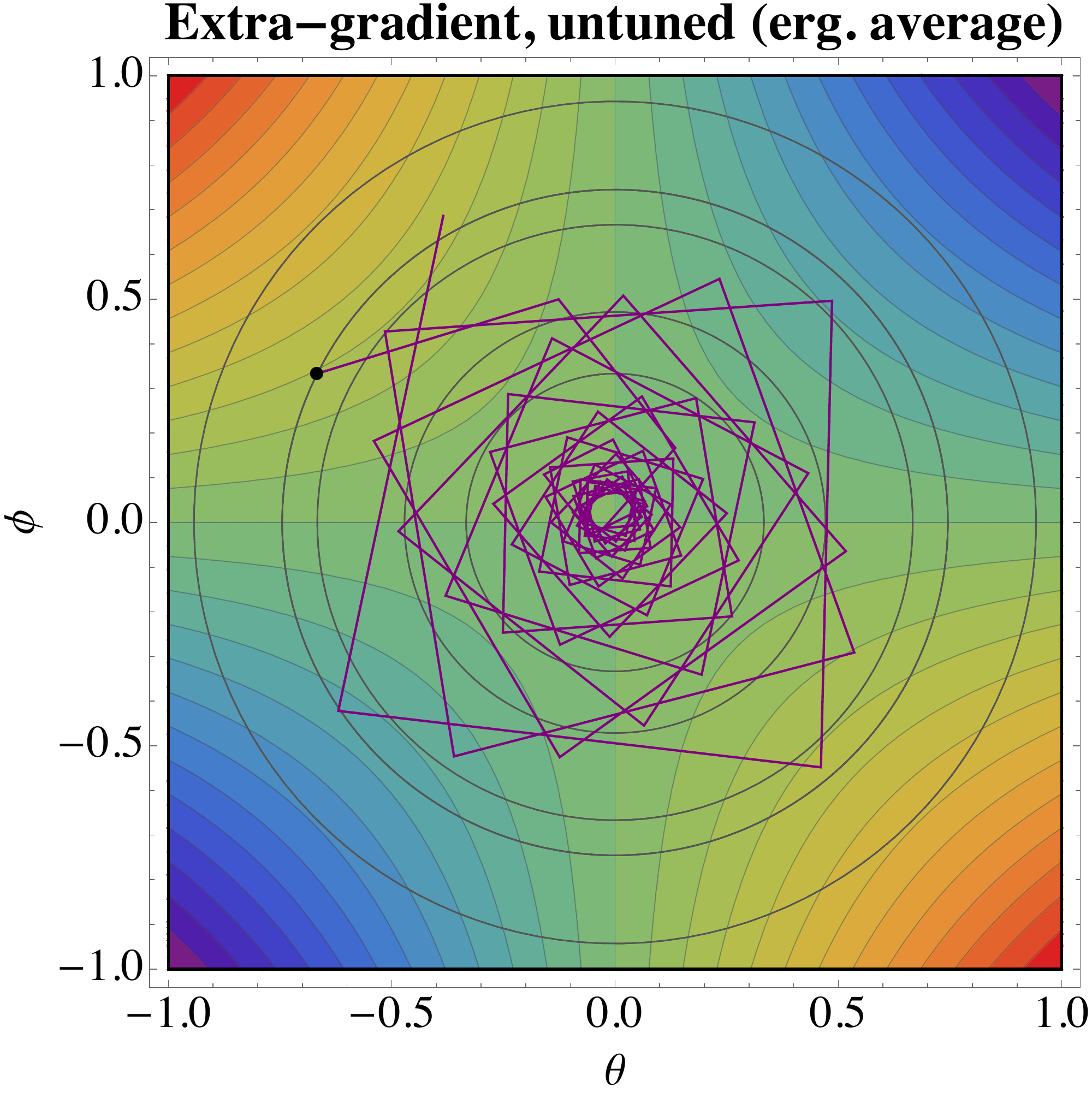}
\hfill
\includegraphics[width=0.3\textwidth]{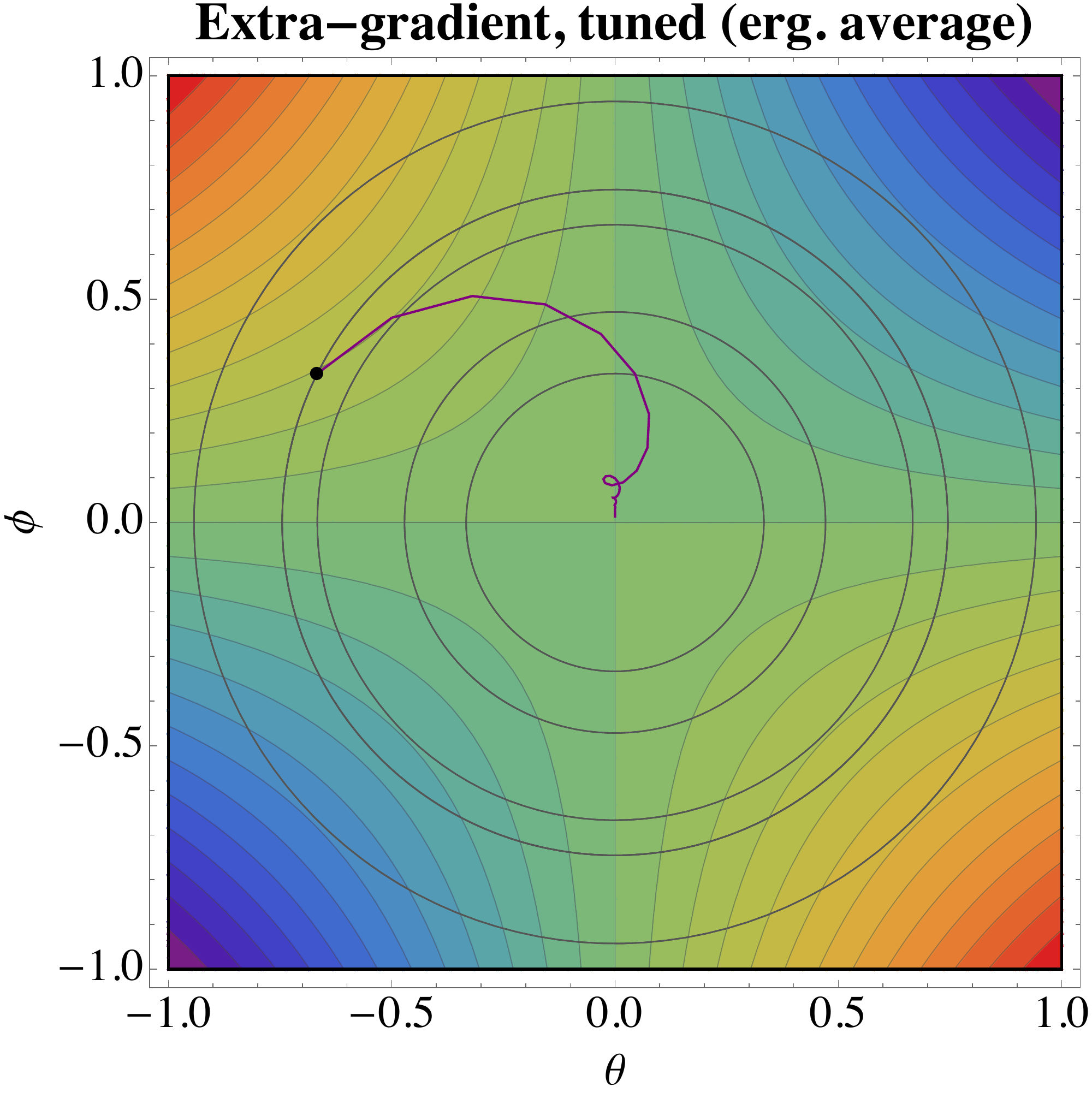}
\hfill
\includegraphics[width=0.3\textwidth]{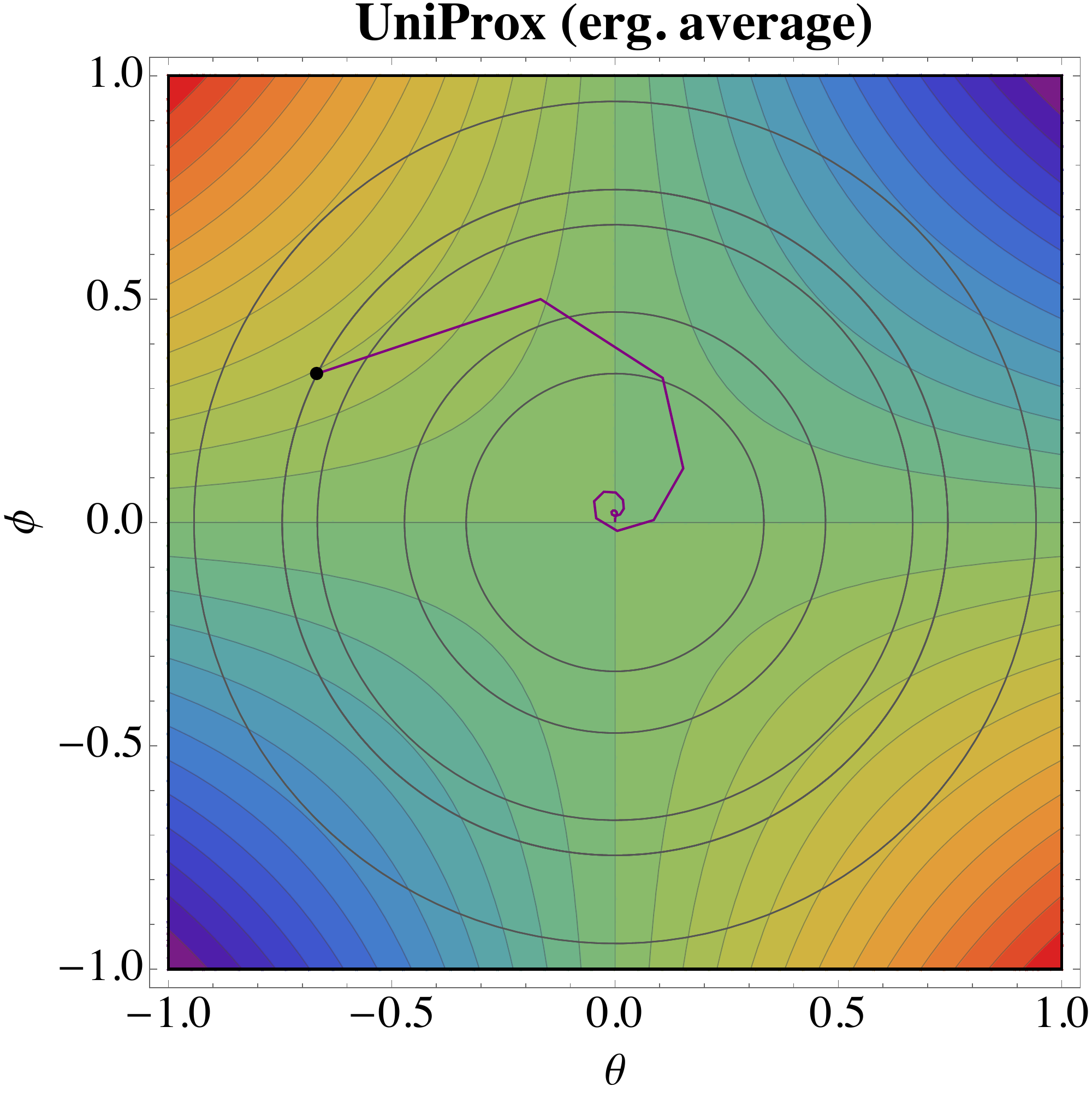}%
\caption{The behavior of \eqref{eq:EG} in the bilinear min-max problem $\minmax(\minvar,\maxvar) = \minvar\maxvar$ with $\minvar,\maxvar \in [-1,1]$.
\revise{Given the clipping at $[-1,1]$, this problem is smooth with $\smooth = 1$;
instead, in the unconstrained case, both \eqref{eq:bounded} and \eqref{eq:LC} fail.
Still, even in the constrained case,}
running \eqref{eq:EG} with a step-size only slightly above the $1/\smooth$ bound ($\smooth = 1$, $\step = 1.04$) results in a dramatic convergence failure (left plot).
Tuning the step-size of \eqref{eq:EG} resolves this problem (center), but a constant step-size makes the algorithm unnecessarily conservative towards the end.
The proposed \method algorithm automatically exploits previous gradient data to perform more informative extra-gradient steps in later ones, thus achieving faster convergence without tuning.}
\label{fig:failure}
\end{figure}


\section{Rate Interpolation: the Euclidean Case}
\label{sec:Eucl}

As a prelude to our main result, we provide in this section an adaptive version of \eqref{eq:EG} that achieves the ``best of both worlds'' in the Euclidean setting of \cref{sec:extragrad}, \ie
an $\bigoh\parens[\big]{1/\sqrt{\nRuns}}$ convergence rate in problems satisfying \eqref{eq:bounded},
and
an $\bigoh(1/\nRuns)$ rate in problems satisfying \eqref{eq:LC}.
Our starting point is the observation that, if the sequence $\curr$ produced by \eqref{eq:EG} converges to a solution of \eqref{eq:VI}, the difference
\begin{equation}
\label{eq:sdiff}
\curr[\sdiff]
	\defeq \norm{\lead[\signal] - \curr[\signal]}
	= \norm{\vecfield(\lead) - \vecfield(\curr)}
\end{equation}
must itself become vanishingly small if $\vecfield$ is (Lipschitz) continuous.
On the contrary, if $\vecfield$ is \emph{discontinuous}, this difference may remain bounded away from zero (consider for example the $L^{1}$ loss $\loss(\point) = \abs{\point}$ near $0$).
Based on this observation, we consider the adaptive step-size policy:
\begin{equation}
\label{eq:step-Eucl}
\next[\step]
	= 1 \Big/ \sqrt{1 + \txs\sum_{\runalt=\start}^{\run} \iter[\sdiff]^{2}}.
\end{equation} 

The intuition behind \eqref{eq:step-Eucl} is as follows:
If $\vecfield$ is not smooth and $\liminf_{\run\to\infty} \revise{\curr[\sdiff]} > 0$, then $\curr[\step]$ will vanish at a $\Theta\parens[\big]{1/\sqrt{\run}}$ rate, which is the optimal step-size schedule for problems satisfying \eqref{eq:bounded} but not \eqref{eq:LC}.
Instead, if $\vecfield$ satisfies \eqref{eq:LC} and $\curr$ converges to a solution $\sol$ of \eqref{eq:VI}, it is plausible to expect that the infinite series $\sum_{\run} \curr[\sdiff]^{2}$ is summable, in which case the step-size $\curr[\step]$ will not vanish as $\run\to\infty$.
Furthermore, since $\curr[\sdiff]$ is defined in terms of successive gradient differences, it automatically exploits the variation of the gradient data observed up to time $\run$, so it can be expected to adjust to the ``local'' Lipschitz constant of $\vecfield$ around a solution $\sol$ of \eqref{eq:VI}.

Our step-size policy and motivation are similar in spirit to the ``predictable sequence'' approach of \cite{RS13-NIPS}.
\PM{Is this ok?}
However, making our reasoning precise (especially the summability of $\sum_{\run} \curr[\sdiff]^{2}$ in the smooth case) involves considerable conceptual and technical difficulties that we present in detail in the supplement.
For now, we only state (without proof) our main result for problems satisfying \eqref{eq:bounded} or \eqref{eq:LC}.

\begin{theorem}
\label{thm:rate-EG}
\revise{Suppose $\vecfield$ satisfies \eqref{eq:MC},}
let $\cvx$ be a compact neighborhood of a solution of \eqref{eq:VI},
and
let $\depth = \sup_{\point\in\cvx} \norm{\init - \point}^{2}$.
If \eqref{eq:EG} is run with the adaptive step-size policy \eqref{eq:step-Eucl}, we have:
\begin{subequations}
\label{eq:rate-EG}
\begin{alignat}{3}
\text{a\upshape)}
	&\;
		\text{If $\vecfield$ satisfies \eqref{eq:bounded}:}
	&\quad
	&\gap_{\cvx}(\avg_{\nRuns})
		= \bigoh\parens*{\frac{\depth + \revise{4\gbound^{3}} + \log(1 + 4\gbound^{2}\nRuns)}{\sqrt{\nRuns}}}.
		\hspace{12ex}
		\label{eq:rate-EG-bounded}
	\\
\text{b\upshape)}
	&\;
		\text{If $\vecfield$ satisfies \eqref{eq:LC}:}
	&\quad
	&\gap_{\cvx}(\avg_{\nRuns})
		= \bigoh\parens*{\depth \big/ \nRuns}.
		\label{eq:rate-EG-smooth}
\end{alignat}
\end{subequations}
\end{theorem}

\beginrev
\cref{thm:rate-EG} (which is proved in the sequel as a special case of \cref{thm:rate-MP}) should be compared to the corresponding results of \citefull{BL19}.
In the non-smooth case, \cite{BL19} provides a bound of the form $\tilde\bigoh(\alpha \gbound D /\sqrt{\nRuns})$ with
$D^{2} = \frac{1}{2} \max_{\point\in\points} \norm{\point}^{2} - \frac{1}{2} \min_{\point\in\points} \norm{\point}^{2}$ (recall that \cite{BL19} only treats problems with a bounded domain),
and 
$\alpha = \max\{\gbound/\gbound_{0},\gbound_{0}/\gbound\}$ where $\gbound_{0}$ is an initial estimate of $\gbound$.
The worst-case value of $\alpha$ is $\bigoh(\gbound)$ when good estimates are not readily available;
in this regard, \eqref{eq:rate-EG-bounded} essentially replaces the $\bigoh(D)$ constant of \citefull{BL19}  by $\bigoh(\gbound)$.
\endedit
Since $D=\infty$ in problems with an unbounded domain, \cref{thm:rate-EG} provides a significant improvement in this regard.

\beginrev
In terms of $\smooth$, the smooth guarantee of \citefull{BL19} is $\tilde\bigoh(\alpha^{2} \smooth D^{2}/\nRuns)$, so the multiplicative constant in the bound also becomes infinite in problems with an unbounded domain.
In our case, $D^{2}$ is replaced by $\depth$ (which is also finite) times an addiitional multiplicative constant which is increasing in $\gbound$ and $\smooth$ (but is otherwise asymptotic, so it is not included in the statement of \cref{thm:rate-EG}).
\endedit
This removes an additional limitation in the results of \citefull{BL19};
going beyond this improvement, in the next sections we drop even the Euclidean regularity requirements \eqref{eq:bounded}/\eqref{eq:LC}, and we provide a corresponding rate interpolation result that does not require either condition.

%
%

\section{Finsler Regularity}
\label{sec:Finsler}

To motivate our analysis outside the setting of \eqref{eq:bounded}/\eqref{eq:LC},
\beginrev
consider the vector field
\begin{equation}
\label{eq:example}
\vecfield_{\resource}(\point)
	= (\capa_{\resource} - \load_{\resource})^{-1}
		+ \lambda \one\{\load_{\resource}>0\},
		\quad
		\resource = 1,\dotsc,\nResources,
\end{equation}
which corresponds to the distributed computing problem of \cref{sec:WE} plus a regularization term designed to limit the activation of computing nodes at low loads.
Clearly, we have $\norm{\vecfield(\point)} \to \infty$ whenever $\point_{i}\to0^{+}$, so \eqref{eq:bounded} and \eqref{eq:LC} both fail (the latter even if $\lambda=0$).
On the other hand, if we consider the ``local'' norm $\norm{\dvec}_{\point,\ast} = \sum_{i=1}^{\vdim} (\capa_{i} - \point_{i}) \, \abs{\dvec_{i}}$, we have $\norm{\vecfield(\point)}_{\point,\ast} \leq \vdim + \lambda \sum_{i=1}^{\vdim} \capa_{i}$, so $\vecfield$ \emph{is bounded relative to $\norm{\cdot}_{\point,\ast}$}.
\endedit
This observation motivates the use of a \emph{local} \textendash\ as opposed to \emph{global} \textendash\ norm,
which we define formally as follows:

\begin{definition}
\label{def:Finsler}
A \emph{Finsler metric} on a convex subset $\points$ of $\vecspace$ is a continuous function $\fins\from\points \times \vecspace\to \R_{+}$  which satisfies the following properties for all $\point \in \points$ and all $\tvec,\alt\tvec \in \vecspace$:
\begin{enumerate}[topsep=0pt,parsep=0pt]
\item
\emph{Subadditivity:}
	$\fins (\point;\tvec + \alt\tvec) \leq \fins(\point;\tvec)+\fins(\point;\alt\tvec)$.
\item
\emph{Absolute homogeneity:}
	$\fins(\point;\lambda\tvec) = \abs{\lambda} \fins(\point;\tvec)$ for all $\lambda \in \R$.
\item
\emph{Positive-definiteness:}
	$\fins(\point;\tvec)\geq 0$ with equality if and only if $\tvec=0$.
\end{enumerate}
Given a Finsler metric on $\points$, the induced \emph{primal\,/\,dual local norms} on $\points$ are respectively defined as
\begin{align}
\label{eq:norm}
\norm{\tvec}_{\point}
	= \fins(\point;\tvec)
	\quad
	\text{and}
	\quad
\norm{\dvec}_{\point,\ast}
	&= \max \setdef{\braket{\dvec}{\tvec}}{\fins(\point;\tvec) = 1}
\end{align}
for all $\point\in\points$ and all $\tvec,\dvec\in\vecspace$.
\revise{We will also say that a Finsler metric on $\points$ is \emph{regular} when $\norm{\dvec}_{\pointalt,\ast} / \norm{\dvec}_{\point,\ast} = 1 + \bigoh(\norm{\pointalt - \point}_{\point})$ for all $\point,\pointalt\in\points$, $\dvec\in\dspace$.}
Finally, for simplicity, we will also assume in the sequel that $\norm{\cdot}_{\point} \geq \nu \norm{\cdot}$ for some $\nu>0$ and all $\point\in\points$ (this last assumption is for convenience only, as the norm could be redefined to $\norm{\cdot}_{\point} \leftarrow \norm{\cdot}_{\point} + \nu \norm{\cdot}$ without affecting our theoretical analysis).
\end{definition}

When $\points$ is equipped with a regular Finsler metric as above, we will say that it is a \emph{Finsler space}.

\begin{example}
\label{ex:fins-global}
Let $\fins(\point;\tvec) = \norm{\tvec}$ where $\norm{\cdot}$ denotes the reference norm of $\points = \vecspace$.
Then the \revise{properties of \cref{def:Finsler}} are satisfied trivially.
\endenv
\end{example}

\begin{example}
\label{ex:norm-1/x}
\beginrev
For a more interesting example of a Finsler structure, consider the set $\points = (0,1]^{\vdim}$ and the metric $\norm{\tvec}_{\point} = \max_{i} \abs{\tvec_{i}}/\point_{i}$, $\tvec\in\R^{\vdim}$, $\point\in\points$.
In this case $\norm{\dvec}_{\point,\ast} = \sum_{i=1}^{\vdim} \point_{i} \abs{\dvec_{i}}$ for all $\dvec\in\R^{\vdim}$, and the only property of \cref{def:Finsler} that remains to be proved is that of regularity.
\endedit
To that end, we have
\begin{equation}
\txs
\norm{\dvec}_{\alt\point,\ast}-\norm{\dvec}_{\point,\ast}
	\leq \sum_{i=1}^{\vdim} \abs{\dvec_{i}}\cdot \abs{\pointalt_{i} - \point_{i}}
	= \sum_{i=1}^{\vdim} \point_{i} \abs{\dvec_{i}}\cdot \abs{\pointalt_{i} - \point_{i}}/\point_{i}
	\leq \norm{\dvec}_{\point,\ast} \cdot \norm{\pointalt - \point}_{\point}.
\end{equation}
Hence, by dividing by $\norm{\dvec}_{\point,\ast}$, we readily get
\(
\norm{\dvec}_{\alt\point,\ast}/ \norm{\dvec}_{\point,\ast}
	\leq 1 + \norm{\point-\pointalt}_{\point}
\)
\ie $\norm{\cdot}_{\point}$ is regular in the sense of \cref{def:Finsler}.
\revise{As we discuss in the sequel, this metric plays an important role for distributed computing problems of the form presented in \cref{sec:WE}.}
\endenv
\end{example}

With all this in hand,
we will say that a vector field $\vecfield\from\points\to\dspace$ is
\begin{enumerate}
\item
\emph{Metrically bounded} if there exists some $\gbound>0$ such that
\begin{equation}
\tag{MB}
\label{eq:MB}
\norm{\vecfield(\point)}_{\point,\ast}
	\leq \gbound
	\quad
	\text{for all $\point \in \points$}.
\end{equation}
\item
\emph{Metrically smooth} if there exists some $\smooth > 0$ such that
\begin{equation}
\label{eq:MS}
\tag{MS}
\norm{\vecfield(\pointalt) - \vecfield(\point)}_{\point,\ast}
	\leq \smooth \norm{\pointalt-\point}_{\pointalt}
	\quad
	\text{for all $\pointalt,\point \in \points$}.
\end{equation}
\end{enumerate}
The notion of metric boundedness/smoothness extends that of ordinary boundedness/Lipschitz continuity to a Finsler context;
\beginrev
note also that, even though neither side of \eqref{eq:MS} is unilaterally symmetric under the change $\point \leftrightarrow \pointalt$, the condition \eqref{eq:MS} as a whole \emph{is}.
\endedit
Our next example
shows that this extension is \emph{proper}, \ie \eqref{eq:bounded}/\eqref{eq:LC} may both fail while \eqref{eq:MB}/\eqref{eq:MS} both hold:

\begin{example}
\label{ex:oper-1/x}
\beginrev
Consider the change of variables $\point_{\resource} \rightsquigarrow 1 - \point_{\resource}/\capa_{\resource}$ in the resource allocation problem of \cref{sec:WE}.
Then, writing $\vecfield_{\resource}(\point) = - (1/\point_{\resource}) - \lambda \one\{\point_{\resource} < 1\}$ for the transformed field \eqref{eq:example} under this change of variables,
we readily get $\vecfield_{\resource}(\point) \to -\infty$ as $\point_{i}\to0^{+}$;
\endedit
as a result, both \eqref{eq:bounded} and \eqref{eq:LC} fail to hold for \emph{any} global norm on $\vecspace$.
Instead, under the \emph{local} norm $\norm{\tvec}_{\point} = \max_{i} \abs{\tvec}_{i}/\point_{i}$, we have:
\beginrev
\begin{enumerate}
[leftmargin=2em]
\item
For all $\lambda\geq0$, $\vecfield$ satisfies \eqref{eq:MB} with $\gbound = \vdim (1+\lambda)$:
$\norm{\vecfield(\point)}_{\point,\ast} \leq \sum_{i=1}^{\vdim} \point_{i} \cdot (1/\point_{i} + \lambda) = \vdim(1+\lambda)$.
\item
For $\lambda=0$, $\vecfield$ satisfies \eqref{eq:MS} with $\smooth = \vdim$:
indeed, for all $\point, \pointalt \in \points$, we have
\endedit
\begin{equation}
\norm{\vecfield(\pointalt)-\vecfield(\point)}_{\point,\ast}
	= \insum_{i=1}^{\vdim} \point_{i} \abs*{\frac{1}{\pointalt_{i}}-\frac{1}{\point_{i}}}
	= \insum_{i=1}^{\vdim} \frac{\abs{\pointalt_{i} - \point_{i}}}{\pointalt_{i}}
	\leq \vdim \max\nolimits_{i} \frac{\abs{\pointalt_{i} - \point_{i}}}{\pointalt_{i}}
	= \vdim \norm{\pointalt-\point}_{\point'}.
\end{equation}
\end{enumerate}
\end{example}

\section{The \method Algorithm and its Guarantees}
\label{sec:method}

\para{The method}

We are now in a position to define a family of algorithms that is capable of interpolating between the optimal smooth/non-smooth convergence rates for solving \eqref{eq:VI} without requiring either \eqref{eq:bounded} or \eqref{eq:LC}.
To do so, the key steps in our approach will be to
\begin{enumerate*}
[(\itshape i\hspace*{.5pt}\upshape)]
\item
equip $\points$ with a suitable Finsler structure (as in \cref{sec:Finsler});
and
\item
replace the Euclidean projection in \eqref{eq:EG} with a suitable ``Bregman proximal'' step that is compatible with the chosen Finsler structure on $\points$.
\end{enumerate*}

We begin with the latter (assuming that $\points$ is equipped with an arbitrary Finsler structure):

\begin{definition}
\label{def:Breg}
We say that $\hreg:\vecspace \to \R \cup \{\infty\}$ is a \emph{Bregman-Finsler function} on $\points$ if:
\begin{enumerate}[leftmargin=3em,parsep=0pt,topsep=0pt]
\item
$\hreg$ is convex, \ac{lsc}, $\cl(\dom\hreg) = \cl(\points)$, and $\dom\subd\hreg = \points$.
\item
The subdifferential of $\hreg$ admits a \emph{continuous selection} $\nabla \hreg(\point)\in \subd\hreg(\point)$ for all $\point \in \points$.
\item
$\hreg$ is \emph{strongly convex}, \ie there exists some $\hstr >0$ such that
\begin{equation}
\label{eq:strong-local}
\hreg(\pointalt)
	\geq \hreg(\point)
		+ \braket{\nabla \hreg(\point)}{\pointalt-\point}
		+ \tfrac{\hstr}{2}\norm{\pointalt-\point}^{2}_{\point}
\end{equation}
for all $\point \in \points$ and all $\pointalt \in \dom\hreg$.
\end{enumerate}
The \emph{Bregman divergence} induced by $\hreg$ is defined for all $\point \in \points$, $\pointalt \in \dom\hreg$ as
\begin{equation}
\label{eq:Bregman}
\breg(\pointalt,\point)
	= \hreg(\pointalt)
		- \hreg(\point)
		- \braket{\nabla \hreg(\point)}{\pointalt-\point}
\end{equation}
and the associated \emph{prox-mapping} is defined for all $\point \in \points$ and $\dpoint \in \dspace$ as
\begin{equation}
\label{eq:prox}
\prox_{\point}(\dpoint)
	= \argmin\nolimits_{\pointalt\in\points}
		\{ \braket{\dpoint}{\point-\pointalt} + \breg(\pointalt,\point) \}.
\end{equation}
\end{definition}

\cref{def:Breg} is fairly technical, so some clarifications are in order.
First, to connect this definition with the Euclidean setup of \cref{sec:Eucl}, the prox-mapping \eqref{eq:prox} should be seen as the Bregman equivalent of a Euclidean projection step, \ie $\Eucl(\point + \dpoint) \leftrightsquigarrow \prox_{\point}(\dpoint)$.
Second, a key difference between \cref{def:Breg} and other definitions of \revise{Bregman functions} in the literature \citep{Bre67,CT93,BecTeb03,JNT11,NJLS09,Nes07,SS11,Bub15} is that $\hreg$ is assumed strongly convex relative to a \emph{local} norm \textendash\ not a global norm.
This ``locality'' will play a crucial role in allowing the proposed methods to adapt to the geometry of the problem.
For concreteness, we provide below an example that expands further on \cref{ex:norm-1/x,ex:oper-1/x}:

\begin{example}
\label{ex:hreg-1/x}
Consider the local norm $\norm{\tvec}_{\point} = \max_{i} \abs{\tvec_{i}}/\point_{i}$ on $\points=(0,1]^{\vdim}$ and let $\hreg(\point) = \sum_{i=1}^{\vdim} 1/\point_{i}$ on $(0,1]^{\vdim}$.
We then have
\begin{equation}
\breg(\pointalt,\point)
	= \sum_{i=1}^{\vdim} \bracks[\bigg]{
		\frac{1}{\pointalt_{i}} - \frac{1}{\point_{i}} + \frac{\pointalt_{i}-\point_{i}}{\point_{i}^{2}}}
	= \sum_{i=1}^{\vdim} \frac{(\pointalt_{i} - \point_{i})^{2}}{\point_{i}^{2} \pointalt_{i}}
	\geq \sum_{i=1}^{\vdim} (1 - \pointalt_{i}/\point_{i})^{2}
	\geq \norm{\pointalt-\point}^{2}_{\point}
\end{equation}
\ie $\hreg$ is $1$-strongly convex relative to $\norm{\cdot}_{\point}$ on $\points$.
\endenv
\end{example}

With all this is in place, the \acl{EG} method can be adapted to our current setting as follows:
\begin{equation}
\tag{\method}
\txs
\begin{alignedat}{3}
\label{eq:uniprox}
\lead
	&= \prox_{\curr}(-\curr[\step]\curr[\signal])
	&\qquad
\curr[\sdiff]
	&= \norm{\lead[\signal] - \curr[\signal]}_{\lead,\ast}
	\\
\next
	&= \prox_{\curr}(-\curr[\step]\lead[\signal])
	&\qquad
\next[\step]
	&= 1 \Big/ \sqrt{1 + \txs\sum_{\runalt=\start}^{\run} \iter[\sdiff]^{2}}
\end{alignedat}
\end{equation}
with
$\curr[\signal] = \vecfield(\curr)$, $\run=\halfrunning$, as in \cref{sec:extragrad}.
In words, this method builds on the template of \eqref{eq:EG} by
\begin{enumerate*}
[(\itshape i\hspace*{.5pt}\upshape)]
\item
replacing the Euclidean projection with a mirror step;
\item
replacing the global norm in \eqref{eq:step-Eucl} with a dual Finsler norm evaluated at the algorithm's leading state $\lead$. 
\end{enumerate*}
The first of these two steps is the main ingredient of the \acdef{MP} algorithm of \citefull{Nem04};
the name ``\method'' has beeen chosen precisely because the proposed method can be seen as a \acl{MP} method that adapts between the smooth and non-smooth regimes.

\para{Convergence speed}

With all this in hand, our main result for \method can be stated as follows:

\begin{theorem}
\label{thm:rate-MP}
\revise{Suppose $\vecfield$ satisfies \eqref{eq:MC},}
let $\cvx$ be a compact neighborhood of a solution of \eqref{eq:VI},
and
set $\depth = \sup_{\point\in\cvx} \breg(\point,\init)$
Then, the \method algorithm enjoys the guarantees:
\begin{subequations}
\label{eq:rate-MP}
\begin{alignat}{3}
\label{eq:rate-MB}
\text{a\upshape)}
	&\;
		\text{If $\vecfield$ satisfies \eqref{eq:MB}:}
	&\quad
	&\gap_{\cvx}(\avg_{\nRuns})
		= \bigoh\parens*{\frac{\depth + \revise{\gbound^{3}}(1 + 1/\hstr)^{2} + \log(1 + 4\gbound^{2}(1 + 2/\hstr)^{2}\nRuns)}{\sqrt{\nRuns}}}.
	\\
\label{eq:rate-MS}
\text{b\upshape)}
	&\;
		\text{If $\vecfield$ satisfies \eqref{eq:MS}:}
	&\quad
	&\gap_{\cvx}(\avg_{\nRuns})
		= \bigoh\parens*{\depth \big/ \nRuns}.
\end{alignat}
\end{subequations}
\end{theorem}

%
%


\beginrev
For the constants that appear in \cref{eq:rate-MP}, we refer the reader to the discussion following \cref{thm:rate-EG} (of course, since \cref{thm:rate-EG} is a special case of \cref{thm:rate-MP}, it is not surprising that the same remarks apply).
\endedit
As for the proof of \cref{thm:rate-MP}, it is quite intricate, so we defer it to the paper's supplement.
We only mention here that its key element is the determination of the asymptotic behavior of the adaptive step-size policy $\step_{\run}$ in the non-smooth and smooth regimes, \ie under \eqref{eq:MB} and \eqref{eq:MS} respectively.
At a very high level, \eqref{eq:MB} guarantees that the difference sequence $\curr[\sdiff]$ is bounded, which implies in turn that $\sum_{\run=\start}^{\nRuns} \curr[\step] = \Omega(\sqrt{\nRuns})$ and eventually yields the bound \eqref{eq:rate-MB} for the algorithm's ergodic average $\avg_{\nRuns}$.
On the other hand, if \eqref{eq:MS} kicks in, we have the following finer result:

\begin{lemma}
Assume $\vecfield$ satisfies \eqref{eq:MS}.
Then,
\begin{enumerate*}
[\itshape a\upshape)]
\item
$\curr[\step]$ decreases monotonically to a strictly positive limit $\step_{\infty} = \lim_{\run\to\infty} \curr[\step] > 0$;
and
\item
the sequence $\curr[\sdiff]$ is square summable:
in particular, $\sum_{\run=\start}^{\infty} \curr[\sdiff]^{2} = 1/\step_{\infty}^{2} - 1$.
\end{enumerate*}
\end{lemma}


By means of this lemma (which we prove in the paper's supplement), it follows that $\sum_{\run=\start}^{\nRuns} \step_{\run} \geq \step_{\infty} \nRuns = \Omega(\nRuns)$.
Because the algorithm's rate of convergence is controlled by this quantity, it ultimately follows that \method enjoys an $\bigoh(1/\nRuns)$ rate of convergence under \eqref{eq:MS}.
However, the details of the ensuing calculations are quite complicated, so we defer them to the supplement.

\para{Trajectory convergence}
\beginrev

In complement to \cref{thm:rate-MP}, we also provide a trajectory convergence result that governs the \emph{actual} iterates of the \method algorithm:
 
\begin{theorem}
\label{thm:last-MP}
Suppose that $\braket{\vecfield(\point)}{\point - \sol} < 0$ whenever $\sol$ is a solution of \eqref{eq:VI} and $\point$ is not.
If, in addition, $\vecfield$ satisfies \eqref{eq:MB} or $\eqref{eq:MS}$, the iterates $\state_{\run}$ of \method converge to a solution of \eqref{eq:VI}.
\end{theorem}

The importance of this result is that, in many practical applications (especially in non-monotone problems), it is more common to harvest the ``last iterate'' of the method ($\curr$) rather than its ergodic average ($\avg_{\nRuns}$);
as such, \Cref{thm:last-MP}
provides a certain justification for this design choice.
\endedit

The proof of \cref{thm:last-MP} relies on non-standard arguments, so we relegate it to the supplement.
Structurally, the first step is to show that $\state_{\run}$ visits any neighborhood of a solution point $\sol\in\sols$ infinitely often (this is where the coherence assumption $\braket{\vecfield(\point)}{\point - \sol}$ is used).
The second is to use this trapping property in conjunction with a suitable ``energy inequality'' to establish convergence via the use of a quasi-Fejér technique as in \cite{Com01};
this part is detailed in a separate appendix.
\endedit

\beginrev
\section{Numerical Experiments}
\label{sec:numerics}

We conclude in this section with a numerical illustration of the convergence properties of \method in two different settings:
\begin{enumerate*}
[\itshape a\upshape)]
\item
bilinear min-max games;
and
\item
a simple Wasserstein \ac{GAN} in the spirit of \citet{DISZ18} with the aim of learning an unknown covariance matrix.
\end{enumerate*}


\begin{figure}[t]
\centering
\includegraphics[height=28ex]{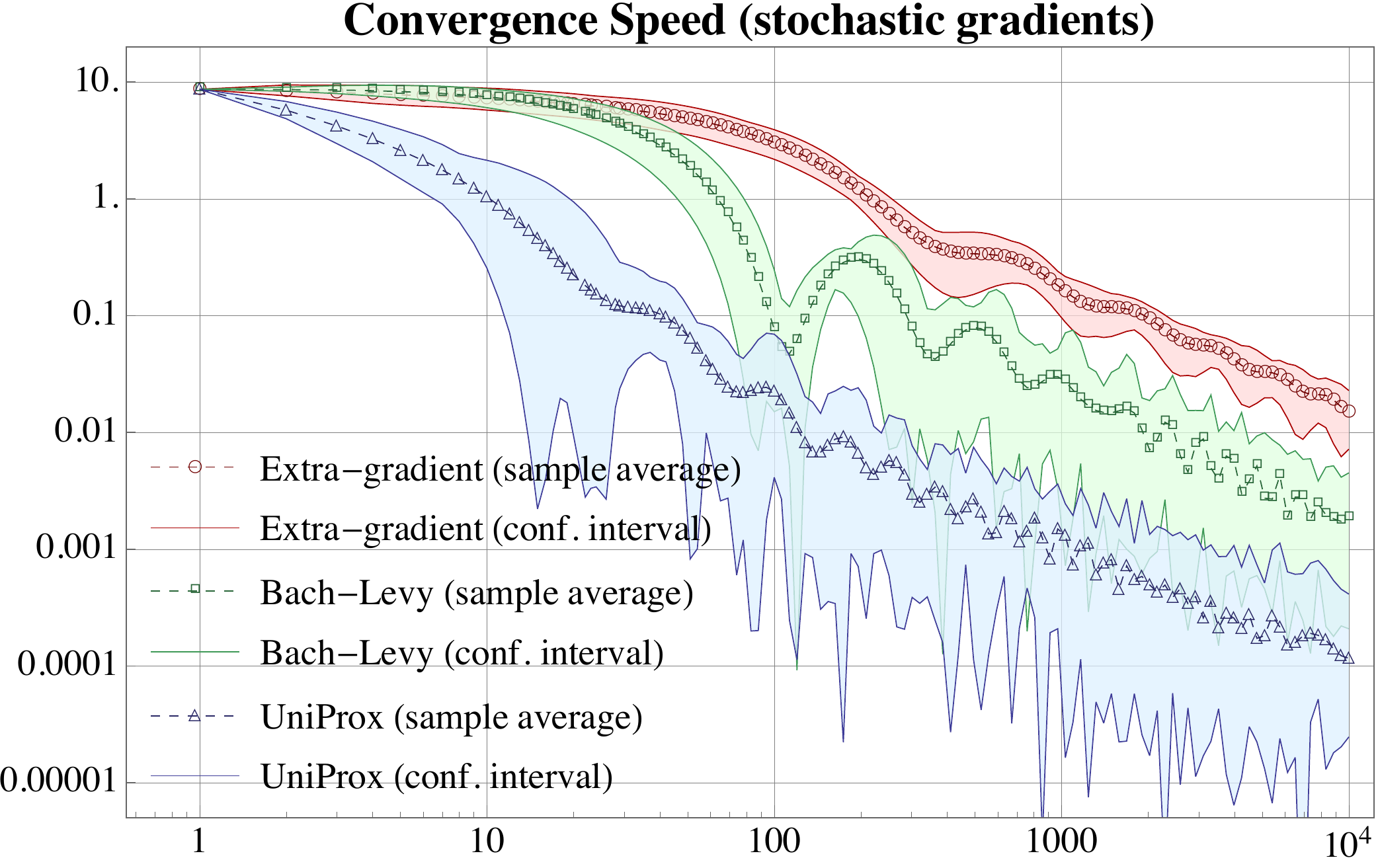}
\hfill
\includegraphics[height=28ex]{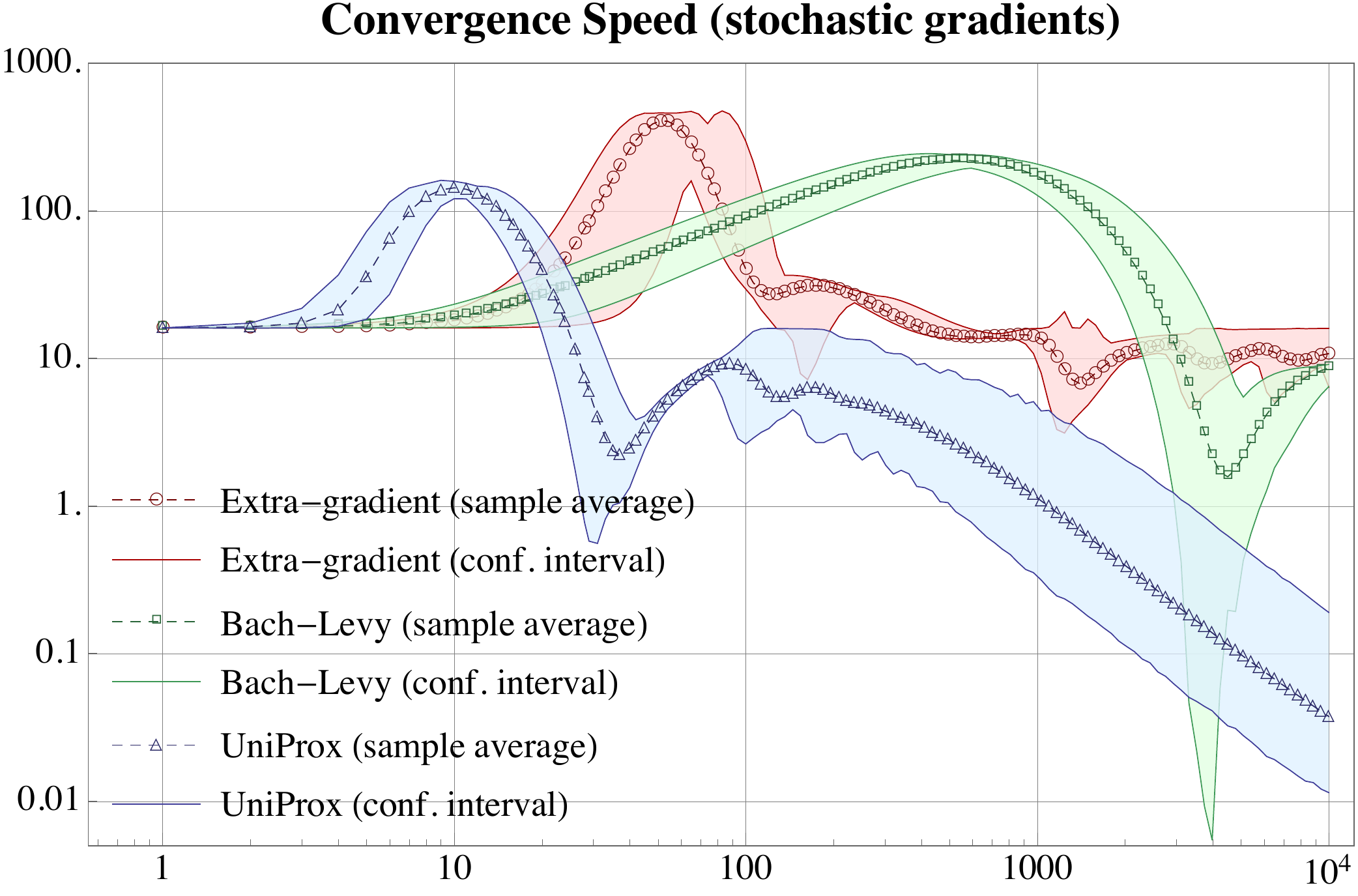}%
\caption{\beginrev
Numerical comparison between the \acf{EG}, \acf{BL} and \method algorithms (red circles, green squares and blue triangles respectively).
The figure on the left shows the methods' convergence in a $100\times 100$ bilinear game;
the one on the right shows the methods' convergence in a non-convex/non-concave covariance learning problem.
In both cases,
the parameters of the \ac{EG} and \ac{BL} algorithms have been tuned with a grid search (\method has no parameters to tune).
All curves have been averaged over $S=100$ sample runs, and the $95\%$ confidence interval is indicated by the shaded area.}
\label{fig:numerics}
\vspace{-3ex}
\end{figure}


\para{Bilinear min-max games}
For our first set of experiments, we consider a min-max game of the form of the form $\minmax(\minvar,\maxvar) = (\minvar - \minsol)^{\top} \mat (\maxvar - \maxsol)$ with $\minvar,\maxvar\in\R^{100}$ and $\mat \in \R^{100}\times\R^{100}$ (drawn \acs{iid} component-wise from a standard Gaussian).
To test the convergence of \method beyond the ``full gradient'' framework, we ran the algorithm with stochastic gradient signals of the form $\curr[\signal] = \vecfield(\curr) + \curr[\noise]$ where $\curr[\noise]$ is drawn \acs{iid} from a centered Gaussian distribution with unit covariance matrix.
We then plotted in \cref{fig:numerics} the squared gradient norm $\norm{\vecfield(\avg_{\nRuns})}^{2}$ of the method's ergodic average $\avg_{\nRuns}$ after $\nRuns$ iterations (so values closer to zero are better).
For benchmarking purposes, we also ran the \acf{EG} and \acf{BL} algorithms \cite{BL19} with the same random seed for the simulated gradient noise.
The step-size parameter of the \ac{EG} algorithm was chosen as $\curr[\step] = 0.025/\sqrt{\run}$, whereas the \ac{BL} algorithm was run with diameter and gradient bound estimation parameters $D_{0} = .5$ and $M_{0} = 2.5$ respectively (both determined after a hyper-parameter search since the only \emph{theoretically} allowable values are $D_{0} = M_{0} = \infty$; interestingly, very large values for $D_{0}$ and $M_{0}$ did not yield good results).
The experiment was repeated $S=100$ times, and \method gave consistently faster rates.

\para{Covariance matrix learning}
Going a step further, we also considered the covariance learning game
\begin{equation}
\label{eq:covGAN}
\minmax(\minvar,\maxvar)
	= \ex_{x\sim\mathcal{N}(0,\Sigma)}\bracks{x^{\top} \minvar x}
	- \ex_{z\sim\mathcal{N}(0,I)}\bracks{z^{\top}\minvar^{\top}\maxvar\minvar z},
	\qquad
	\minvar,\maxvar\in\R^{\vdim}\times\R^{\vdim}.
\end{equation}
The goal here is to generate data drawn from a centered Gaussian distribution with unknown covariance $\Sigma$;
in particular, this model follows the Wasserstein \ac{GAN} formulation of \citet{DISZ18} with generator and discriminator respectively given by $G(z) = \minvar z$ and $D(x) = x^{\top}\maxvar x$ (no clipping).
For the experiments, we took $\vdim=100$, a mini-batch of $m=128$ samples per update, and we ran the \ac{EG}, \ac{BL} and \method algorithms as above, tracing the square norm of $\vecfield$ as a measure of convergence.
Since the problem is non-monotone, there are several disjoint equilibrium components so the algorithms' behavior is considerably more erratic;
however, after this initial warm-up phase, \method again gave the faster convergence rates.

\endedit

\section*{Acknowledgments}
%
%
This research was partially supported by the COST Action CA16228 ``European Network for Game Theory'' (GAMENET)
and the French National Research Agency (ANR) in the framework of
the grants ORACLESS (ANR–16–CE33–0004–01) and ELIOT (ANR-18-CE40-0030),
the ``Investissements d'avenir'' program (ANR-15-IDEX-02),
the LabEx PERSYVAL (ANR-11-LABX-0025-01), 
and
MIAI@Grenoble Alpes (ANR-19-P3IA-0003).

\appendix
\numberwithin{equation}{section}		
\numberwithin{lemma}{section}		
\numberwithin{proposition}{section}		
\numberwithin{theorem}{section}		

\section{Properties of the restricted gap function}
\label{app:aux}

In this appendix, we discuss the basic properites of the restricted merit function $\gap_{\cvx}$ introduced in \eqref{eq:gap}.
For completeness, we provide the proof of \cref{prop:gap},which itself is an extension of a similar result by \citefull{Nes07}:

\begin{proof}[Proof of \cref{prop:gap}]
Let $\sol\in \points$ be a solution of \eqref{eq:VI} so $\braket{\oper(\sol)}{\point-\sol}
\geq 0$ for all $\point\in \points$.
Then, by monotonicity, we get:
\begin{align}
\braket{\oper(\point)}{\sol-\point}
	&\leq \braket{\oper(\point)-\oper(\sol)}{\sol-\point}
		+\braket{\oper(\sol)}{\sol-\point}
	\notag\\
	&= -\braket{\oper(\sol)-\oper(\point)}{\sol-\point}
		-\braket{\oper(\sol)}{\point-\sol}
	\leq 0,
\end{align}
so $\gap_{\cvx}(\sol) \leq 0$.
On the other hand, if $\sol\in\cvx$, we also get $\gap(\sol) \geq \braket{\oper(\sol)}{\sol - \sol} = 0$, so we conclude that $\gap_{\cvx}(\sol) = 0$.

For the converse statement, assume that $\gap_{\cvx}(\test) = 0$ for some $\test\in\cvx$ and suppose that $
\cvx$ contains a neighborhood of $\test$ in $\points$.
First, we claim that the following inequality holds:
\begin{equation}
\label{eq:localMVI}
\braket{\oper(\point)}{\point-\test}
	\geq 0
	\quad
	\text{for all $\point\in \cvx$}.
\end{equation}
Indeed, assume to the contrary that there exists some $\point_{1}\in \cvx$ such that
\begin{equation}
\braket{\oper(\point_{1})}{\point_{1}-\test}
	< 0.
\end{equation}
This would then give
\begin{equation}
0
	= \gap_{\cvx}(\test)
	\geq \braket{\oper(\point_{1})}{\test-\point_{1}}
	> 0,
\end{equation}
which is a contradiction.
Now, we further claim that $\test$ is a solution of \eqref{eq:VI},\ie:
\begin{equation}
\text{$\braket{\oper(\test)}{\point-\test}\geq 0$\;for all $\point\in \points$.}
\end{equation}
If we suppose that there exists some $z_{1}\in \points$ such that $\braket{\oper(\test)}{z_{1}-\test}< 0$, then, by the continuity of $\oper$, there exists a neighborhood $\nhd'$ of $\test$ in $\points$ such that
\begin{equation}
\braket{\oper(\point)}{z_{1}-\point}
	< 0
	\quad
	\text{for all $\point \in \nhd'$}.
\end{equation} 
Hence, assuming without loss of generality that $\nhd'\subset \nhd \subset\cvx$ (the latter assumption due to the assumption that $\cvx$ contains a neighborhood of $\test$),
and taking $\lambda>0$ sufficiently small so that $\point=\test+\lambda(z_{1}-\test)\in \nhd'$, we get that
$\braket{\oper(\point)}{\point-\test}=\lambda\braket{\oper(\point)}{z_{1}-\test}<0$, in contradiction to \eqref{eq:localMVI}.
We conclude that $\test$ is a solution of \eqref{eq:VI}, as claimed.
\end{proof}
\section{Properties of Bregman functions and proximal mappings}
\label{app:Bregman}


In this appendix, we present some basic facts about Bregman functions and proximal mappings.
Similar results exist in the literature in different contexts (see \eg \cite{Nes07,Nes09,JNT11} and references therein), but given that many of our results rely on the use of \emph{local} \textendash\ as opposed to \emph{global} \textendash\ norms, we provide here complete statements and proofs.
We then have the following basic lemma connecting the above notions:

\begin{lemma}
\label{lem:Bregman}
Let $\hreg$ be a Bregman function on $\points$.
Then, for all $\base\in\dom\hreg$, $\point \in \dom\subd\hreg$ and all $\dpoint \in \subd\hreg(\point)$, we have
\begin{equation}
\label{eq:descent}
\braket{\nabla \hreg(\point)}{\point-\base}
	\leq \braket{\dpoint}{\point-\base}.
\end{equation}
\end{lemma}

\begin{proof}
By a simple continuity argument, it is sufficient to show that the inequality holds for the 
relative interior $\relint\points$ of $\points$.
In order to show this, pick a base point $\base\in\relint\points$, and let
\begin{equation}
\phi(t)
	= \hreg(\point+t(\base-\point))
	- [\hreg(\point)+\braket{\dpoint}{t(\base-\point)}]
	\quad
	\text{for all $t\in[0,1]$}.
\end{equation}
Since, $\hreg$ is strongly convex and $\dpoint\in \partial \hreg(\point)$ due to the first equivalence, it follows
that $\phi(t)\geq 0$ with equality if and only if $t=0$. Since, $\psi(t)=\braket{\nabla \hreg(\point+
t(\base-x))-\dpoint}{\base-\point}$ is a continuous selection of subgradients of $\phi$ and both $\phi$ and
$\psi$ are continuous over $[0,1]$, it follows that $\phi$ is continuously differentiable with $\phi'=\psi$
on $[0,1]$. Hence, with $\phi$ convex and $\phi(t)\geq 0=\phi(0)$ for all $t \in [0,1]$, we conclude that
$\phi'(0)=\braket{\nabla \hreg(x)-\dpoint}{\base-\point}\geq 0$ and thus we obtain the result.
\end{proof}

The basic ingredient for establishing connections in the Bregman framework is a generalization of the rule of cosines which is known in the literature as the ``three-point identity'' \citep{CT93} and will be the main tool for
deriving the main estimations for our analysis. Being more precise, we have the following lemma:

\begin{lemma}
\label{lem:3point}
Let $\hreg$ be a Bregman function on $\points$. Then, for all $\base \in \points$ and all $\point,\pointalt \in \subpoints$, we have:
\begin{equation}
\breg(\base,\pointalt)= \breg(\base,\point)+\breg(\point,\pointalt)+\braket{\nabla \hreg(\pointalt)-\nabla \hreg(\point)}{\point-\base}
\end{equation}
\end{lemma}


The proof of this lemma follows as in the classic Bregman case \cite{CT93} so we omit it and proceed to derive some key bounds for the Bregman divergence before and after a mirror step:

\begin{proposition}
\label{eq:prop1}
Let $\hreg$ be a local Bregman function with strong convexity modulus $\hstr >0$.
Fix some $\base \in \points$ and let $\new\point=\prox_{\point}(\dvec)$ for some $\point \in \subpoints$ and $\dvec\in \dspace$.
We then have:
\begin{equation}
\label{eq:descent2}
\breg(\base,\new\point)\leq \breg(\base,\point)-\breg(\new\point,\point)+\braket{\dvec}{\new\point-\base}\end{equation}
\end{proposition}

\begin{proof}
By the three-point identity established in \cref{lem:3point}, we get:
\begin{equation}
\breg(\base,\point)=\breg(\base,\new\point)+\breg(\new\point,\point)+\braket{\nabla 
\hreg(\point)-\nabla \hreg(\new\point)}{\new\point-\base}
\end{equation}
By rearranging the terms we get:
\begin{equation}
\breg(\base,\new\point)=\breg(\base,\point)-\breg(\new\point,\point)+\braket{\nabla 
\hreg(\new\point)-\nabla \hreg(\point)}{\new\point-\base}
\end{equation}
Due to \eqref{eq:descent} and the fact that $\new\point=\prox_{\point}(\dvec)$ so $\nabla \hreg(\point)+\dvec \in \partial \hreg (\new\point)$, we get the result.
\end{proof}

Thanks to the above estimations, we obtain the following inequalities relating the Bregman divergence
between \emph{two} prox-steps:

\begin{proposition}
\label{prop:2prox}
Let $\hreg$ be a Bregman function compatible on $\points$. 
Letting $\new\point_{1}=\prox_{\point}(\dvec_{1})$ and $\new\point_{2}=\prox_{\point}(\dvec_{2})$, we have:
\begin{subequations}
\begin{align}
\breg(\base,\new\point_{2})
	&\leq \breg(\base,\point)
		+ \braket{\dvec_{2}}{\new\point_{1}-\base}
		+\bracks{\braket{\dvec_{2}}{\new\point_{2} - \new\point_{1}} - \breg(\new\point_{2},\point)}
	\\
\label{eq:2}
	&\leq \breg(\base,\point)
		+ \braket{\dvec_{2}}{\new\point_{1} - \base}
		+ \braket{\dvec_{2} - \dvec_{1}}{\new\point_{2} - \new\point_{1}}
		- \breg(\new\point_{2},\new\point_{1}) - \breg(\new\point_{1},\point).
\end{align}
\end{subequations}
\end{proposition}

\begin{proof}
For the first inequality, by applying \cref{eq:prop1} for $\new\point_{2}=\prox_{\point}(\dvec_{2})$, we get:
\begin{align}
\breg(\base,\new\point_{2})
	&\leq \breg(\base,\point)-\breg(\new\point_{2},\point)
	+ \braket{\dvec_{2}}{\new\point_{2}-\base}
	\notag\\
	&= \breg(\base,\point)
	+ \braket{\dvec_{2}}{\new\point_{1}-\base}
	+ \bracks{\braket{\dvec_{2}}{\new\point_{2} - \new\point_{1}} - \breg(\new\point_{2},\point)}
\end{align}
For the second inequality, we need to bound $\braket{\dvec_{2}}{\new\point_{2}-\new\point_{1}}-
\breg_{\hreg}(\new\point_{2},\point)$. In particular,  applying again \cref{eq:prop1} for $\base=\new\point_{2}$, we get:
\begin{equation}
\breg(\new\point_{2},\new\point_{1})\leq \breg(\new\point_{2},\point)+\braket{\dvec_{1}}{\new\point_{1}-\new\point_{2}}-\breg(\new\point_{1},\point)
\end{equation}
and hence:
\begin{equation}
\breg(\new\point_{2},\point)\geq \breg(\new\point_{2},\new\point_{1})
	+ \breg(\new\point_{1},\point)
	- \braket{\dvec_{1}}{\new\point_{1} - \new\point_{2}}.
\end{equation}
So, combining the above inequalities we get:
\begin{equation}
\braket{\dvec_{2}}{\new\point_{2}-\new\point_{1}}-\breg(\new\point_{2},\point)\leq
\braket{\dvec_{2}}{\new\point_{2}-\new\point_{1}}
-\breg(\new\point_{2},\new\point_{1})-\breg(\new\point_{1},\point)-\braket{\dvec_{1}}{\new\point_{2}-\new\point_{1}}
\end{equation}
and thus we get the second inequality as well.
\end{proof}

\section{Main bounds and  energy inequality}
\label{app:Bounded}

In this appendix, we shall provide the bound of the variation of the operators, \ie
\begin{equation}
\norm{\oper(\state_{\run+1/2})-\oper(\state_{\run})}^{2}_{\state_{\run},\ast}
\end{equation}
that lies in the core of our analysis. 
To begin with, we recall that $(\points,\norm{\cdot}_{\point})$ is a regular Finsler space, \ie
$\norm{\dvec}_{\point,\ast}/\norm{\dvec}_{\alt \point,\ast}=1+\bigoh(\norm{\point-\pointalt}_{\point})$.
However, in what follows we shall assume the more general condition:
\begin{equation}
\label{eq:Finsler-regular}
\text{$\norm{\dvec}_{\point,\ast}/\norm{\dvec}_{\alt \point,\ast}\leq 1+\beta \left[ \norm{\point-\alt \point}_{\point}+\norm{\point-\alt \point}_{\alt \point} \right]$ for some $\beta >0$}
\end{equation}
\begin{remark}
It is straightforward for one to observe that a regular Finsler space satisfies \eqref{eq:Finsler-regular} for $\beta=1$.
\end{remark}
Owning this regularity geometrical property for the problem's domain we shall proceed into showing that 
\begin{equation}
\norm{\oper(\state_{\run+1/2})-\oper(\state_{\run})}^{2}_{\state_{\run},\ast}
\end{equation}
is uniformly bounded. More precisely, we have the following lemma.

\begin{lemma}
\label{lem:bounded residual}
Suppose that $\oper$ satisfies \eqref{eq:MB}. Then, the sequence $\norm{\oper(\state_{\run+1/2})-\oper(\state_{\run})}^{2}_{\state_{\run},\ast}$ is bounded. In particular, the following inequality holds:
\begin{equation}
\norm{\oper(\state_{\run+1/2})-\oper(\state_{\run})}^{2}_{\state_{\run},\ast}\leq C^{2}
\end{equation}
with $C=2\gbound+ \beta\frac{4\gbound}{\hstr}$.
\end{lemma}

\begin{proof}
It suffices to show that: $\norm{\oper(\state_{\run+1/2}) - \oper(\state_{\run})}_{\state_{\run+1/2},\ast}$
is bounded. More precisely, by the triangle inequality we have:

\begin{equation}
\label{eq:triangle}
\norm{\oper(\state_{\run+1/2}) - \oper(\state_{\run})}_{\state_{\run+1/2},\ast}\leq \norm{\oper(\state_{\run+1/2})}_{\state_{\run+1/2},\ast}+\norm{\oper(\state_{\run})}_{\state_{\run+1/2},\ast}
\end{equation}

Let us now bound the (RHS) part of \eqref{eq:triangle} term by term. In particular, we have:

\begin{itemize}
\item
For the first term $\norm{\oper(\state_{\run+1/2})}_{\state_{\run+1/2},\ast}$ we readily get due to \eqref{eq:MB}:

\begin{equation}
\label{eq:bound1}
\norm{\oper(\state_{\run+1/2})}_{\state_{\run+1/2},\ast}\leq \gbound
\end{equation}

\item
For the second term $\norm{\oper(\state_{\run})}_{\state_{\run+1/2},\ast}$, we have:
\begin{multline}
\label{eq:bound1}
\norm{\oper(\state_{\run})}_{\state_{\run+1/2},\ast}\leq   \norm{\oper(\state_{\run})}_{\state_{\run},\ast}
+\beta \left[ \norm{\state_{\run}-\state_{\run+1/2}}_{\state_{\run}}+\norm{\state_{\run}-\state_{\run+1/2}}_{\state_{\run+1/2}} \right]
\\
\leq  \gbound+\beta
\left[ \norm{\state_{\run}-\state_{\run+1/2}}_{\state_{\run}}+\norm{\state_{\run}-\state_{\run+1/2}}_{\state_{\run+1/2}} \right]
\end{multline}
Therefore, it suffices to show that the quantity $\norm{\state_{\run}-\state_{\run+1/2}}_{\state_{\run}}+\norm{\state_{\run}-\state_{\run+1/2}}_{\state_{\run+1/2}}$ is bounded from above. Indeed, we have:
\begin{align*}
\breg(\state_{\run},\state_{\run+1/2})+\breg(\state_{\run+1/2},\state_{\run})&=\braket{\nabla \hreg(\state_{\run})-\nabla \hreg(\state_{\run+1/2})}{\state_{\run}-\state_{\run+1/2}}\\
&\leq \step_{\run}\braket{\oper(\state_{\run})}{\state_{\run}-\state_{\run+1/2}}\\
&\leq \gbound \step_{\run} \norm{\state_{\run}-\state_{\run+1/2}}_{\state_{\run}}
\end{align*}
where the last inequality is obtained due to \eqref{eq:MB}.
Moreover, due to \eqref{eq:strong-local} we get:
\begin{align*}
\breg(\state_{\run},\state_{\run+1/2})+\breg(\state_{\run+1/2},\state_{\run})
&\leq \step_{\run} \gbound \sqrt{\frac{2}{\hstr}\breg(\state_{\run+1/2},\state_{\run})}\\
  &\leq \gbound \sqrt{\frac{2}{\hstr}\left[\breg(\state_{\run},\state_{\run+1/2})+\breg(\state_{\run+1/2},\state_{\run})\right]}
  \end{align*}
 which yields
  \begin{equation}
 \breg(\state_{\run},\state_{\run+1/2})+\breg(\state_{\run},\state_{\run+1/2})\leq \frac{2\gbound^{2}}{\hstr}  \end{equation}
Hence, due to the local strong convexity  \eqref{eq:strong-local} of $\hreg$, we get:
\begin{equation}
\frac{\hstr}{2}\left[ \norm{\state_{\run}-\state_{\run+1/2}} _{\state_{\run}}^{2}+\norm{\state_{\run}-\state_{\run+1/2}} _{\state_{\run+1/2}}^{2}\right]\leq \frac{2\gbound^{2}}{\hstr}
\end{equation}

which in turn implies that:
\begin{equation}
\text{$\norm{\state_{\run}-\state_{\run+1/2}}_{\state_{\run}}\leq \frac{2\gbound}{\hstr}$ and $\norm{\state_{\run}-\state_{\run+1/2}}_{\state_{\run+1/2}}\leq \frac{2\gbound}{\hstr}$}
\end{equation}
and so,
\begin{equation}
\label{eq:bound2}
\norm{\state_{\run}-\state_{\run+1/2}}_{\state_{\run}}+\norm{\state_{\run}-\state_{\run+1/2}}_{\state_{\run+1/2}}\leq \frac{4\gbound}{\hstr}
\end{equation}

Moreover, by combining \eqref{eq:bound1} and \eqref{eq:bound2} we get:
\begin{equation}
\label{eq:final-bound2}
\norm{\oper(\state_{\run})}_{\state_{\run+1/2},\ast}\leq  \gbound +\beta \frac{4\gbound}{\hstr}
\end{equation}
\end{itemize} 
Summarizing, \eqref{eq:triangle} combined with \eqref{eq:bound1} and \eqref{eq:final-bound2} yields:

\begin{equation}
\label{eq:bounded variation}
\norm{\oper(\state_{\run+1/2})-\oper(\state_{\run})}_{\state_{\run+1/2},\ast}\leq 2\gbound+ \beta\frac{4\gbound}{\hstr}
\end{equation}

and hence the result follows.
\end{proof}

We now proceed to prove the energy inequality stated in \cref{lem:energy}.

\begin{lemma}
\label{lem:energy}
For all $\point\in\points$, the iterates $\state_{\run}$ of  \method satisfy the recursive bound:
\begin{multline}
\label{eq:energy}
\breg(\point,\state_{\run+1})
	\leq \breg(\point,\state_{\run})
		- \step_{\run}\braket{\oper(\state_{\run+1/2})}{\state_{\run+1/2}-\point}
		+\step_{\run}\braket{\oper(\state_{\run+1/2})-\oper(\state_{\run})}{\state_{\run+1}-\state_{\run+1/2}}
		\\
		-\breg(\state_{\run+1},\state_{\run+1/2})-\breg(\state_{\run+1/2},\state_{\run})
\end{multline}  

\end{lemma}

\begin{proof}
The result follows directly by setting  $\state_{1}^{+}=\state_{\run+1/2}$, $\state_{2}^{+}=\state_{\run+1}$, $\point=\state_{\run}$, $\dvec_{1}=-\step_{\run}\oper(\state_{\run})$ and $\dvec_{2}=-\step_{\run}\oper(\state_{\run+1/2})$ in \cref{prop:2prox}.
\end{proof}

\section{Rate interpolation guarantees}
\label{app:universal}

In this appendix, we provide the proof of the the regime-agnostic rate interpolation guarantees of the UniProx. In order, to provide the necessary the respective rates we shall provide an intermediate result concerning the case of \eqref{eq:MS}. Formally, we have the following lemma.

\begin{lemma}
\label{lem:summability}
Assume $\oper$ satisfies \eqref{eq:MS} and $\state_{\run}, \state_{\run+1/2}$ are the iterates of  \method. Then, the following hold:
\begin{enumerate}
\item
$\step_{\run}\to \inf_{\run \in \N}\step_{\run}=\step_{\infty}>0$
\item
The sequence $\norm{\oper(\state_{\run+1/2})-\oper(\state_{\run})}^{2}_{\state_{\run+1/2},\ast}$ is summable. In particular, we have:
\begin{equation}
\sum_{\run=1}^{+\infty}\norm{\oper(\state_{\run+1/2})-\oper(\state_{\run})}^{2}_{\state_{\run+1/2},\ast}=
\frac{1}{\step_{\infty}^{2}}-1
\end{equation}
\end{enumerate} 
\end{lemma}
\begin{proof}
Since $\step_{\run}$ is decreasing and bounded from below ($\step_{\run}\geq 0$), then we readily obtain that its limit exists and more precisely we have:
\begin{equation}
\lim_{\run \to +\infty}\step_{\run}=\inf_{\run \in \N}\step_{\run}=\step_{\infty} \geq 0
\end{equation}
Let us now assume that  $\step_{\infty}=0$. Then, by recalling \eqref{eq:energy}:
\begin{multline}
\breg(\base,\state_{\run+1})\leq \breg(\base,\state_{\run})-\step_{\run} \braket{\oper(\state_{\run+1/2})}{\state_{\run+1/2}-\base}+\step_{\run}\braket{\oper(\state_{\run+1/2})-\oper(\state_{\run})}{\state_{\run+1}-\state_{\run+1/2}}
\\
-\breg(\state_{\run+1/2},\state_{\run})-\breg(\state_{\run+1},\state_{\run+1/2})
\end{multline}
By rearranging the above and telescoping $\run=1,\dotsc, \nRuns$ we get:
\begin{multline}
\label{eq:gap-rate2}
\sum_{\run=1}^{\nRuns}\step_{\run}\braket{\oper(\state_{\run+1/2})}{\state_{\run+1/2}-\base}\leq \breg(\base,\state_1)+\sum_{\run=1}^{\nRuns}\step_{\run}\braket{\oper(\state_{\run+1/2})-\oper(\state_{\run})}{\state_{\run+1}-\state_{\run+1/2}}
\\
-\sum_{\run=1}^{\nRuns}\breg(\state_{\run+1/2},\state_{\run})-\sum_{\run=1}^{\nRuns}\breg(\state_{\run+1},\state_{\run+1/2})
\end{multline}
whereas, by applying Fenchel-Young inequality to the above we readily get:
\begin{multline}
\label{eq:gap-rate3}
\sum_{\run=1}^{\nRuns}\step_{\run}\braket{\oper(\state_{\run+1/2})}{\state_{\run+1/2}-\base}\leq \breg(\base,\state_1)+\frac{1}{2\hstr}\sum_{\run=1}^{\nRuns}\step_{\run}^{2}\norm{\oper(\state_{\run+1/2})-\oper(\state_{\run})}^{2}_{\state_{\run+1/2},\ast}
\\
+\frac{\hstr}{2}\sum_{\run=1}^{\nRuns}\norm{\state_{\run+1}-\state_{\run+1/2}}_{\state_{\run+1/2}}^{2}
-\sum_{\run=1}^{\nRuns}\breg(\state_{\run+1/2},\state_{\run})-\sum_{\run=1}^{\nRuns}\breg(\state_{\run+1},\state_{\run+1/2})
\end{multline}
and by considering that by \eqref{eq:strong-local}:
\begin{equation}
\frac{\hstr}{2}\sum_{\run=1}^{\nRuns}\norm{\state_{\run+1}-\state_{\run+1/2}}^{2}_{\state_{\run+1/2}}-\sum_{\run=1}^{\nRuns}\breg(\state_{\run+1},\state_{\run+1/2})\leq 0
\end{equation}
we finally obtain:
\begin{multline}
\sum_{\run=1}^{\nRuns}\step_{\run}\braket{\oper(\state_{\run+1/2})}{\state_{\run+1/2}-\base}\leq \breg(\base,\state_1)+\frac{1}{2\hstr}\sum_{\run=1}^{\nRuns}\step_{\run}^{2}\norm{\oper(\state_{\run+1/2})-\oper(\state_{\run})}^{2}_{\state_{\run+1/2},\ast}
\\
-\sum_{\run=1}^{\nRuns}\breg(\state_{\run+1/2},\state_{\run})
\end{multline}
Therefore,  by the definition \eqref{eq:MS} we have:
\begin{multline}
\sum_{\run=1}^{\nRuns}\step_{\run}\braket{\oper(\state_{\run+1/2})}{\state_{\run+1/2}-\base}\leq \breg(\base,\state_1)+\frac{1}{2\hstr}\sum_{\run=1}^{\nRuns}\step_{\run}^{2}\norm{\oper(\state_{\run+1/2})-\oper(\state_{\run})}^{2}_{\state_{\run+1/2},\ast}
\\
-\frac{\hstr}{2\smooth^{2}}\sum_{\run=1}^{\nRuns}\norm{\oper(\state_{\run+1/2})-\oper(\state_{\run})}_{\state_{\run+1/2},\ast}^{2}
\end{multline}
which becomes:
\begin{multline}
\sum_{\run=1}^{\nRuns}\step_{\run}\braket{\oper(\state_{\run+1/2})}{\state_{\run+1/2}-\base}\leq \breg(\base,\state_1)+\sum_{\run=1}^{\nRuns}\left[\frac{\step_{\run}^{2}}{2\hstr}-\frac{\hstr}{4\smooth^{2}}\right]\norm{\oper(\state_{\run+1/2})-\oper(\state_{\run})}^{2}_{\state_{\run+1/2},\ast}
\\
-\frac{\hstr}{4\smooth^{2}}\sum_{\run=1}^{\nRuns}\norm{\oper(\state_{\run+1/2})-\oper(\state_{\run})}_{\state_{\run+1/2},\ast}^{2}
\end{multline}
Now, by setting $\base=\sol$ with $\sol$ being a solution of \eqref{eq:VI} and using the fact that $\braket{\oper(\state_{\run+1/2})}{\state_{\run+1/2}-\sol}\geq 0$ and $\breg(\sol,\state_1)\leq D'$ (by the compatibility of $\hreg$), we obtain:
\begin{equation}
\label{eq:sum}
\frac{\hstr}{4\smooth^{2}}\sum_{\run=1}^{\nRuns}\norm{\oper(\state_{\run+1/2})-\oper(\state_{\run})}^{2}_{\state_{\run+1/2},\ast}\leq D'+\sum_{\run=1}^{\nRuns}\left[\frac{\step_{\run}^{2}}{2\hstr}-\frac{\hstr}{4\smooth^{2}}\right]\norm{\oper(\state_{\run+1/2})-\oper(\state_{\run})}^{2}_{\state_{\run+1/2},\ast}\end{equation}
Moreover, by observing that the quantity $\left[\frac{\step_{\run}^{2}}{2\hstr}-\frac{\hstr}{4\smooth^{2}}\right] \leq 0$, whenever $\step_{\run}\leq \sqrt{2}\hstr /2\smooth$ and since we assumed that $\step_{\run}\to 0$,
there exists some $\run_{0}\in \N$ such that:
\begin{equation}
\left[\frac{\step_{\run}^{2}}{2\hstr}-\frac{\hstr}{4\smooth^{2}}\right] \leq 0\;\;\text{for all}\;\;\run\geq \run_0
\end{equation}
Therefore, \eqref{eq:sum} becomes:
\begin{equation}
\frac{1}{\step_{\nRuns+1}}-1=\sum_{\run=1}^{\nRuns}\norm{\oper(\state_{\run+1/2})-\oper(\state_{\run})}^{2}_{\state_{\run+1/2},\ast}\leq D'+\sum_{\run=1}^{\run_{0}}\left[\frac{\step_{\run}^{2}}{2\hstr}-\frac{\hstr}{4\smooth^{2}}\right]\norm{\oper(\state_{\run+1/2})-\oper(\state_{\run})}^{2}_{\state_{\run+1/2},\ast}\end{equation}
In addition, since $1/\step_{\nRuns+1}\to +\infty$, by the fact that $\step_{\run}\to 0$, this yields that:
\begin{equation}
+\infty \leq  D'+\sum_{\run=1}^{\run_{0}}\left[\frac{\step_{\run}^{2}}{2\hstr}-\frac{\hstr}{4\smooth^{2}}\right]\norm{\oper(\state_{\run+1/2})-\oper(\state_{\run})}^{2}_{\state_{\run+1/2},\ast}
\end{equation}
which is a contradiction. Hence, we get that:
\begin{equation}
\lim_{\run \to +\infty}\step_{\run}=\inf_{\run \in \N}\step_{\run}=\step_{\infty}>0
\end{equation}
In order to prove our second claim, we first recall the definition of $\step_{\run}$:
\begin{equation}
\step_{\run}=\frac{1}{\sqrt{1+\sum_{j=1}^{\run-1}\norm{\oper(\state_{\run+1/2})-\oper(\state_{\run})}_{\state_{\run+1/2,\ast}}^{2}}}
\end{equation}
whereas by developing and rearranging we have:
\begin{equation}
\sum_{j=1}^{\run-1}\norm{\oper(\state_{\run+1/2})-\oper(\state_{\run})}_{\state_{\run+1/2,\ast}}^{2}=\frac{1}{\step_{\run}^{2}}-1
\end{equation}
Hence, by taking limits on both sides we get:
\begin{equation}
\sum_{\run=1}^{+\infty}\norm{\oper(\state_{\run+1/2})-\oper(\state_{\run})}_{\state_{\run+1/2,\ast}}^{2}=\lim_{\run \to +\infty}\sum_{j=1}^{\run-1}\norm{\oper(\state_{\run+1/2})-\oper(\state_{\run})}_{\state_{\run+1/2,\ast}}^{2}=\frac{1}{\step_{\infty}^{2}}-1
\end{equation}
where $0\leq\frac{1}{\step_{\infty}^{2}}-1<+\infty$, since $0<\step_{\infty} \leq 1$ and therefore the result follows.
\end{proof}

\begin{proof}[Proof of \cref{thm:rate-MP}]
By recalling \eqref{eq:energy} we have:
\begin{multline}
\breg(\base,\state_{\run+1})\leq \breg(\base,\state_{\run})-\step_{\run} \braket{\oper(\state_{\run+1/2})}{\state_{\run+1/2}-\base}+\step_{\run}\braket{\oper(\state_{\run+1/2})-\oper(\state_{\run})}{\state_{\run+1}-\state_{\run+1/2}}
\\
-\breg(\state_{\run+1/2},\state_{\run})-\breg(\state_{\run+1},\state_{\run+1/2})
\end{multline}
We start our analysis rearranging \eqref{eq:energy}. In particular, by telescoping $\run=1,\dotsc, \nRuns$ we get:
\begin{multline}
\label{eq:gap-rate2}
\sum_{\run=1}^{\nRuns}\step_{\run}\braket{\oper(\state_{\run+1/2})}{\state_{\run+1/2}-\base}\leq \breg(\base,\state_1)+\sum_{\run=1}^{\nRuns}\step_{\run}\braket{\oper(\state_{\run+1/2})-\oper(\state_{\run})}{\state_{\run+1}-\state_{\run+1/2}}
\\
-\sum_{\run=1}^{\nRuns}\breg(\state_{\run+1/2},\state_{\run})-\sum_{\run=1}^{\nRuns}\breg(\state_{\run+1},\state_{\run+1/2})
\end{multline}
On the other hand, since $\oper$ is monotone, we readily get:
\begin{equation}
\label{eq:monotone}
\step_{\run}\braket{\oper(\base)}{\state_{\run+1/2}-\base}\leq \step_{\run}\braket{\oper(\state_{\run+1/2})}{\state_{\run+1/2}-\base}
\end{equation}
Thus, combining \eqref{eq:monotone} and \eqref{eq:gap-rate2}, dividing by $\sum_{\run=1}^{\nRuns}\step_{\run}$ and setting $\bar\state_{\nRuns}=\left[\sum_{\run=1}^{\nRuns}\step_{\run}\right]^{-1}\sum_{\run=1}^{\nRuns}\step_{\run}\state_{\run+1/2}$ we get:
\begin{multline}
\braket{\oper(\base)}{\bar\state_{\nRuns}-\base}\leq \breg(\base,\state_1)+\sum_{\run=1}^{\nRuns}\step_{\run}\braket{\oper(\state_{\run+1/2})-\oper(\state_{\run})}{\state_{\run+1}-\state_{\run+1/2}}-\sum_{\run=1}^{\nRuns}\breg(\state_{\run+1/2},\state_{\run})
\\
-\sum_{\run=1}^{\nRuns}\breg(\state_{\run+1},\state_{\run+1/2})
\end{multline}
whereas, by applying Fenchel-Young inequality to the above we readily get:
\begin{multline}
\braket{\oper(\base)}{\bar\state_{\nRuns}-\base}\leq \breg(\base,\state_1)+\frac{1}{2\hstr}\sum_{\run=1}^{\nRuns}\step_{\run}^{2}\norm{\oper(\state_{\run+1/2})-\oper(\state_{\run})}^{2}_{\state_{\run+1/2},\ast}+\frac{\hstr}{2}\sum_{\run=1}^{\nRuns}\norm{\state_{\run+1}-\state_{\run+1/2}}_{\state_{\run+1/2}}^{2}
\\
-\sum_{\run=1}^{\nRuns}\breg(\state_{\run+1/2},\state_{\run})-\sum_{\run=1}^{\nRuns}\breg(\state_{\run+1},\state_{\run+1/2})
\end{multline}
Thus, if $\cvx$ is a compact neighbourhood of the solution set $\sols$, considering that by \eqref{eq:strong-local}:
\begin{equation}
\frac{\hstr}{2}\sum_{\run=1}^{\nRuns}\norm{\state_{\run+1}-\state_{\run+1/2}}^{2}_{\state_{\run+1/2}}-\sum_{\run=1}^{\nRuns}\breg(\state_{\run+1},\state_{\run+1/2})\leq 0
\end{equation}
 and taking suprema on both sides, yields:
\begin{multline}
\label{eq:gap-rate3}
\gap_{\cvx}(\bar\state_{\nRuns})\leq \left[\sum_{\run=1}^{\nRuns}\step_{\run}\right]^{-1}(\sup_{\base\in \cvx}\breg(\base,\state_1)+\frac{1}{2\hstr}\sum_{\run=1}^{\nRuns}\step_{\run}^{2}\norm{\oper(\state_{\run+1/2})-\oper(\state_{\run})}^{2}_{\state_{\run+1/2},\ast}
\\
-\sum_{\run=1}^{\nRuns}\breg(\state_{\run+1/2},\state_{\run}))
\end{multline}

\para{Case 1: Convergence under \eqref{eq:MB}}
Therefore, in order to determine the convergence speed of $\overline{\state}_{\nRuns}$ under \eqref{eq:MB}, we shall examine the asymptotic behaviour of each term of the nominator on the (RHS) of \eqref{eq:gap-rate3}.
In particular, we have the following:
\begin{itemize}
\item
For the first term:  we readily get by  the compactness of $\cvx$,
\begin{equation}
\label{eq:first term}
\text{$\sup_{\base \in \cvx}\breg(\base,\state_1)\leq \breg'$ for some  constant $\breg'>0$.}
\end{equation}
by the compatibility of the regularizer $\hreg$.
\item
For the second term: $\sum_{\run=1}^{\nRuns}\step_{\run}^{2}\norm{\oper(\state_{\run+1/2})-\oper(\state_{\run})}_{\state_{\run+1/2}}^{2}$, we have:
\begin{multline}
\sum_{\run=1}^{\nRuns}\step_{\run}^{2}\norm{\oper(\state_{\run+1/2})-\oper(\state_{\run})}_{\state_{\run+1/2}}^{2}=\sum_{\run=1}^{\nRuns}(\step_{\run}^{2}-\step_{\run+1}^{2})\norm{\oper(\state_{\run+1/2})-\oper(\state_{\run})}_{\state_{\run+1/2}}^{2}
\\
+\sum_{\run=1}^{\nRuns}\step_{\run+1}^{2}\norm{\oper(\state_{\run+1/2})-\oper(\state_{\run})}_{\state_{\run+1/2}}^{2}
\end{multline}
Hence, $\step_{\run}$ is non-increasing and therefore $(\step_{\run}^{2}-\step_{\run+1}^{2}\geq 0)$,  and $\step_{\run}\leq 1$ the above becomes:
\begin{align*}
\sum_{\run=1}^{\nRuns}\step_{\run}^{2}\norm{\oper(\state_{\run+1/2})-\oper(\state_{\run})}_{\state_{\run+1/2}}^{2}&\leq C^{2}+\sum_{\run=1}^{\nRuns}\step_{\run+1}\norm{\oper(\state_{\run+1/2})-\oper(\state_{\run})}_{\state_{\run+1/2}}^{2}
\\
&\leq C^{2}+\sum_{\run=1}^{\nRuns}\frac{\norm{\oper(\state_{\run+1/2})-\oper(\state_{\run})}_{\state_{\run+1/2}}^{2}}{1+\sum_{j=1}^{\run}\norm{\oper(\state_{\run+1/2})-\oper(\state_{\run})}_{\state_{\run+1/2}}^{2}}
\\
&\leq C^{2}+1+\log(1+\sum_{\run=1}^{\nRuns}\norm{\oper(\state_{\run+1/2})-\oper(\state_{\run})}_{\state_{\run+1/2}}^{2})
\end{align*}
with the last inequality being obtained by \cref{lem:logarithmic-1} which combined with \eqref{eq:MB} yields:
\begin{equation}
\label{eq:logarithmic-residual}
\sum_{\run=1}^{\nRuns}\step_{\run}^{2}\norm{\oper(\state_{\run+1/2})-\oper(\state_{\run})}_{\state_{\run+1/2}}^{2}\leq C^{2}+1+\log(1+C^{2}\nRuns)
\end{equation}
\end{itemize}

Finally, for $\sum_{\run=1}^{\nRuns}\step_{\run}$, we have the following upper-bound
\begin{equation}
\sum_{\run=1}^{\nRuns}\step_{\run}=\sum_{\run=1}^{\nRuns}\frac{1}{\sqrt{1+\sum_{j=1}^{\run-1}\norm{\oper(\state_{\run+1/2})-\oper(\state_{\run})}_{\state_{\run+1/2,\ast}}^{2}}}\geq \sum_{\run=1}^{\nRuns}\frac{1}{\sqrt{1+\run C^{2}}}
\end{equation}
which yields:
\begin{equation}
\label{eq:third term}
\sum_{\run=1}^{\nRuns}\step_{\run}=\Omega(\sqrt{\nRuns})\;\; \text{and}\;\; \sum_{\run=1}^{\nRuns}\step_{\run}\to +\infty
\end{equation}
Now, by combining \eqref{eq:first term}, \eqref{eq:logarithmic-residual} and \eqref{eq:third term} we readily get that under \eqref{eq:MB} we get that:
\begin{equation}
\gap_{\cvx}(\bar\state_{\nRuns})=\bigoh(1/\sqrt{\nRuns}).
\end{equation}

\para{Case 2: Convergence under \eqref{eq:MS}}
We now suppose that $\oper$ satisfies \eqref{eq:MS} condition. By applying \cref{lem:summability} along with :
\begin{multline}
\label{eq:gap-rate3}
\sum_{\run=1}^{\nRuns}\step_{\run}\braket{\oper(\state_{\run+1/2})}{\state_{\run+1/2}-\base}\leq \breg(\base,\state_1)+\frac{1}{2\hstr}\sum_{\run=1}^{\nRuns}\step_{\run}^{2}\norm{\oper(\state_{\run+1/2})-\oper(\state_{\run})}^{2}_{\state_{\run+1/2},\ast}
\\
-\sum_{\run=1}^{\nRuns}\breg(\state_{\run+1/2},\state_{\run})
\end{multline}
 by examining the asymptotic behaviour  term by term, we get:
\begin{itemize}
\item
For the first term $\breg(\sol,\state_1)$, since $\sol \in \dom\oper=\dom \hreg$ and $\state_1 \in \dom \partial \hreg$, we have:
\begin{equation}
\label{eq:bounded Bregman}
\breg(\sol,\state_1)<+\infty
\end{equation}
\item
For the second term $\sum_{\run=1}^{\nRuns}\step_{\run}^{2}\norm{\oper(\state_{\run+1/2})-\oper(\state_{\run})}_{\state_{\run+1/2},\ast}^{2}$we have:
\begin{equation}
\sum_{\run=1}^{\nRuns}\step_{\run}^{2}\norm{\oper(\state_{\run+1/2})-\oper(\state_{\run})}_{\state_{\run+1/2},\ast}^{2}\leq \sum_{\run=1}^{\nRuns}\norm{\oper(\state_{\run+1/2})-\oper(\state_{\run})}_{\state_{\run+1/2},\ast}^{2}
\end{equation}
and by applying \cref{lem:summability} we have:
\begin{equation}
\sum_{\run=1}^{\nRuns}\step_{\run}^{2}\norm{\oper(\state_{\run+1/2})-\oper(\state_{\run})}_{\state_{\run+1/2},\ast}^{2}\leq  \frac{1}{\step_{\infty}^{2}}-1
\end{equation}
with $\step_{\infty}=\inf_{\run}\step_{\run}>0$.
\end{itemize}
Finally, by applying \cref{lem:summability} once more by considering $\step_{\infty}=\inf_{\run \in \N}\step_{\run}>0$ we have:
\begin{equation}
\sum_{\run=1}^{\nRuns}\step_{\run}\geq \step_{\infty} \sum_{\run=1}^{\nRuns}1=\step_{\infty} \nRuns
\end{equation}
which yields:
\begin{equation}
\label{eq:linear growth}
\sum_{\run=1}^{\nRuns}\step_{\run}=\Omega (\nRuns)
\end{equation}
and the result follows.
\end{proof}

\section{Last iterate's convergence analysis}
\label{app:iterate}

In this appendix, we establish the convergence of the sequence generated by \eqref{eq:uniprox}, \ie its so-called last iterate. In particular, we show that the actual iterates (before averaging) of \method converge towards the solution set $\sols$.This result comprises of two parts: first we extract convergent subsequences of $\state_{\run},\state_{\run+1/2}$ to the said set; then we apply the "trapping" argument described in \cref{sec:method} .

\begin{lemma}
\label{lem:last1}
Suppose that $\oper$ satisfies \eqref{eq:MB} (respectively \eqref{eq:MS})  and $\state_{\run},\state_{\run+1/2}$ are the iterates of \method. Then, the following hold:
\begin{enumerate}
\item
\label{eq:global-conv}
\text{$\norm{\state_{\run+1/2}-\state_{\run}}\to 0$ while $\run \to +\infty$}
\item
$\max\{\breg(\state_{\run+1/2},\state_{\run}), \breg(\state_{\run},\state_{\run+1/2})\}\leq \frac{2\gbound^{2}}{\hstr}\step_{\run}^{2}$
\end{enumerate}
\end{lemma}

\begin{proof}
For the proof of the first claim, we shall treat the cases of \eqref{eq:MB} and \eqref{eq:MS} individually.
\para{Case 1: Under \eqref{eq:MB} condition}
Since $\step_{\run}$ is decreasing and bounded from below, then we readily obtain that. its limit exists and more precisely:
\begin{equation}
\lim_{\run \to +\infty}\step_{\run}=\step_{\infty} \geq 0
\end{equation}
We shall distinguish two individual cases:
\begin{itemize}
\item
\emph{\bf{$\step_{\infty}>0$:} }By recalling the definition of the adaptive step-size:
\begin{equation}
\step_{\run}=\frac{1}{\sqrt{1+\sum_{j=1}^{\run-1}\norm{\oper(\state_{j+1/2})-\oper(\state_j)}_{\state_{j+1/2}}^{2}}}
\end{equation}
whereas by rearranging and developing we have:
\begin{equation}
\sum_{j=1}^{\run-1}\norm{\oper(\state_{j+1/2})-\oper(\state_{j})}_{\state_{j+1/2}}^{2}=\frac{1}{\step_{\run}^{2}}-1
\end{equation}
Therefore, by taking limits on both sides:
\begin{equation}
\label{eq:sum1}
\sum_{\run=1}^{+\infty}\norm{\oper(\state_{j+1/2})-\oper(\state_{j})}_{\state_{j+1/2}}^{2}=\lim_{\run \to +\infty}
\frac{1}{\step_{\run}^{2}}-1=\frac{1}{\step_{\infty}^{2}}-1\geq 0
\end{equation}
Hence, by recalling \eqref{eq:energy} we have:
\begin{align*}
\sum_{\run=1}^{\nRuns}\breg(\state_{\run+1/2},\state_{\run})&\leq \breg(\sol,\state_1)+\sum_{\run=1}^{\nRuns} \step_{\run}^{2}\norm{\oper(\state_{\run+1/2})-\oper(\state_{\run})}_{\state_{\run+1/2}}^{2}
\\ 
&\leq \breg(\sol,\state_1)+\sum_{\run=1}^{\nRuns} \norm{\oper(\state_{\run+1/2})-\oper(\state_{\run})}_{\state_{\run+1/2}}^{2}
\end{align*}
which in turn by \eqref{eq:sum1} yields $\sum_{\run=1}^{+\infty}\breg(\state_{\run+1/2},\state_{\run})<+\infty$
and hence $\breg(\state_{\run+1/2},\state_{\run})\to 0$. Moreover, by applying \eqref{eq:strong-local}:
\begin{equation}
\frac{\hstr}{2}\norm{\state_{\run+1/2}-\state_{\run}}^{2}_{\state_{\run}}\leq \breg(\state_{\run+1/2},\state_{\run})
\end{equation}
 Now, by recalling $\mu\norm{\cdot}\leq  \norm{\cdot}_{\point}$, we get:
\begin{equation}
\norm{\state_{\run+1/2}-\state_{\run}}^{2}\leq \frac{1}{\mu^{2}}\norm{\state_{\run+1/2}-\state_{\run}}^{2}_{\state_{\run}}
\end{equation}
and the result follows.
\item
\emph{\bf{$\step_{\infty}=0$:}}
By the prox-step, we get:
\begin{multline}
\label{eq:product-norm1}
\braket{\nabla \hreg(\state_{\run})-\nabla \hreg(\state_{\run+1/2})}{\state_{\run}-\state_{\run+1/2}}\leq \step_{\run}\braket{\oper(\state_{\run})}{\state_{\run}-\state_{\run+1/2}}
\\    
\leq \step_{\run}\norm{\oper(\state_{\run})}_{\state_{\run},\ast}\norm{\state_{\run}-\state_{\run+1/2}}_{\state_{\run}}
\end{multline}
On the other hand, we have:
\begin{equation}
\braket{\nabla \hreg(\state_{\run})-\nabla \hreg(\state_{\run+1/2})}{\state_{\run}-\state_{\run+1/2}} = \breg(\state_{\run},\state_{\run+1/2})+\breg(\state_{\run+1/2},\state_{\run})
\end{equation}
Thus, we get by \eqref{eq:strong-local}:
\begin{align*}
\breg(\state_{\run},\state_{\run+1/2})+\breg(\state_{\run+1/2},\state_{\run})&\leq \step_{\run}\norm{\oper(\state_{\run})}_{\state_{\run},\ast}\norm{\state_{\run}-\state_{\run+1/2}}_{\state_{\run}}\\
&\leq \step_{\run}\gbound\sqrt{\frac{2}{\hstr}\left[\breg(\state_{\run},\state_{\run+1/2})+\breg(\state_{\run+1/2},\state_{\run}) \right]}
\end{align*}
where the last inequality is obtained due to \eqref{eq:MB}; which in turn yields:
\begin{equation}
\breg(\state_{\run},\state_{\run+1/2})+\breg(\state_{\run+1/2},\state_{\run})\leq \frac{2\gbound^{2}}{\hstr}\step_{\run}^{2}
\end{equation}
So, a fortiori we have:
\begin{equation}
\breg(\state_{\run},\state_{\run+1/2})\leq \frac{2\gbound^{2}}{\hstr}\step_{\run}^{2}
\end{equation}
Moreover, by \eqref{eq:strong-local}:
\begin{equation}
\frac{\hstr}{2}\norm{\state_{\run+1/2}-\state_{\run}}^{2}_{\state_{\run+1/2}}\leq \breg(\state_{\run},\state_{\run+1/2})\leq \frac{2\gbound^{2}}{\hstr}\step_{\run}^{2}
\end{equation}
 Now, by recalling $\mu\norm{\cdot}\leq  \norm{\cdot}_{\point}$, we get:
\begin{equation}
\norm{\state_{\run+1/2}-\state_{\run}}^{2}\leq \frac{1}{\mu^{2}}\norm{\state_{\run+1/2}-\state_{\run}}^{2}_{\state_{\run}}
\end{equation}
and the result follows since we assumed that $\step_{\run}\to 0$. 
\end{itemize}
\para{Case 2: Under \eqref{eq:MS} condition}
Following similar reasoning as above, we have:
\begin{align*}
\sum_{\run=1}^{\nRuns}\breg(\state_{\run+1/2},\state_{\run})&\leq \breg(\sol,\state_1)+\sum_{\run=1}^{\nRuns}\step_{\run}^{2}\norm{\oper(\state_{\run+1/2})-\oper(\state_{\run})}_{\state_{\run+1/2},\ast}^{2}
\\
&\leq \breg(\sol,\state_1)+\sum_{\run=1}^{\nRuns}\norm{\oper(\state_{\run+1/2})-\oper(\state_{\run})}_{\state_{\run+1/2},\ast}^{2}
\end{align*}
which by taking limits on both sides and by applying \cref{lem:summability} we get that:
\begin{equation}
\sum_{\run=1}^{+\infty}\breg(\state_{\run+1/2},\state_{\run})<+\infty
\end{equation}
Therefore, $\breg(\state_{\run+1/2},\state_{\run}) \to 0$, whereas by applying \eqref{eq:strong-local} we obtain:
\begin{equation}
\frac{\hstr}{2}\norm{\state_{\run+1/2}-\state_{\run}}^{2}_{\state_{\run}}\leq \breg(\state_{\run+1/2},\state_{\run})
\end{equation}
 Now, by recalling $\mu\norm{\cdot}\leq  \norm{\cdot}_{\point}$, we get:
\begin{equation}
\norm{\state_{\run+1/2}-\state_{\run}}^{2}\leq \frac{1}{\mu^{2}}\norm{\state_{\run+1/2}-\state_{\run}}^{2}_{\state_{\run}}
\end{equation}
and the result follows.
\\
On the other hand, for the second claim, we have by the prox-step:
\begin{align*}
\breg(\state_{\run},\state_{\run+1/2})+\breg(\state_{\run+1/2},\state_{\run})&\leq \step_{\run} \braket{\oper(\state_{\run})}{\state_{\run+1/2}-\state_{\run}}\\
&\leq \step_{\run}\gbound\norm{\state_{\run+1/2}-\state_{\run}}_{\state_{\run}}
\end{align*}
Therefore, by following the same reasoning with the first claim, we get:
\begin{equation}
\breg(\state_{\run},\state_{\run+1/2})+\breg(\state_{\run+1/2},\state_{\run})\leq \frac{2\gbound^{2}}{\hstr} \step_{\run}^{2}
\end{equation}
and hence since $\breg(\cdot,\cdot)\geq 0$, we have:
\begin{equation}
\breg(\state_{\run+1/2},\state_{\run})\leq \frac{2\gbound^{2}}{\hstr}\step_{\run}^{2}\;\;\text{and}\;\; \breg(\state_{\run},\state_{\run+1/2})\leq \frac{2\gbound^{2}}{\hstr}\step_{\run}^{2}
\end{equation}
and so the result follows
\end{proof}

\begin{remark}
We shall point out that \eqref{eq:global-conv} in \cref{lem:last1} establishes the convergence with respect to the global ambient reference norm of $\R^{\vdim}$.
\end{remark}

\begin{proposition}
\label{prop:sub-seq}
Suppose that $\oper$ satisfies \eqref{eq:MB} (respectively \eqref{eq:MS}).
Then, the iterates $\state_{\run},\state_{\run+1/2}$ of \method possess
convergent subsequences towards the equilibrium set $\sols$.
\end{proposition}

\begin{proof}
By \cref{lem:last1}, it suffices to show that $\state_{\run+1/2}$ possesses such a subsequence. Assume to the contrary that it does not. That implies that:
\begin{equation}
\liminf_{\run}\dist(\state_{\run+1/2},\sols)=\delta >0
\end{equation}
which in turn yields,
\begin{equation}
\liminf_{\run}\braket{\oper(\state_{\run+1/2})}{\state_{\run+1/2}-\sol}=c>0
\end{equation}
Now, by setting $\base=\sol$ for some $\sol \in \sols$ in \eqref{eq:energy}, we get:
\begin{align*}
\breg(\sol,\state_{\run+1})&\leq \breg(\sol,\state_{\run})-\step_{\run}\braket{\oper(\state_{\run+1/2})}{\state_{\run+1/2}-\sol}+\step_{\run}^{2}\norm{\oper(\state_{\run+1/2})-\oper(\state_{\run})}_{\state_{\run+1/2}}^{2}\\
&\leq \breg(\sol,\state_{\run})-c\step_{\run}+\step_{\run}^{2}\norm{\oper(\state_{\run+1/2})-\oper(\state)}_{\state_{\run+1/2}}^{2}
\end{align*}
whereas by telescoping $\run=1,\dotsc ,\nRuns$ we obtain:
\begin{equation}
\label{eq:last2}
\breg(\sol,\state_{\nRuns})\leq \breg(\sol,\state_1)-\sum_{\run=1}^{\nRuns}\step_{\run}\left[ c-\dfrac{\sum_{\run=1}^{\nRuns}\step_{\run}^{2}\norm{\oper(\state_{\run+1/2})-\oper(\state_{\run})}^{2}_{\state_{\run+1/2},\ast}}{\sum_{\run=1}^{\nRuns}\step_{\run}}\right]
\end{equation}

Having this established this general setting, we shall examine the asymptotic behaviour term by term for each regularity case individually, which in both cases shall lead to a contradiction. 

\para{Case 1: Under \eqref{eq:MB} condition}
\begin{itemize}
\item
For the first term: $\sum_{\run=1}^{\nRuns}\step_{\run}$, due to \eqref{eq:uniprox} we have by \eqref{eq:third term} that:
\begin{equation}
\label{eq:first order}
\text{$\sum_{\run=1}^{\nRuns}\step_{\run}\to +\infty$ and $\sum_{\run=1}^{\nRuns}\step_{\run}=\Omega(\sqrt{\nRuns})$}
 \end{equation}
 
\item
For the second term $\dfrac{\sum_{\run=1}^{\nRuns}\step_{\run}^{2}\norm{\oper(\state_{\run+1/2})-\oper(\state_{\run})}^{2}_{\state_{\run+1/2},\ast}}{\sum_{\run=1}^{\nRuns}\step_{\run}}$, we first examine the denominator. In particular, due to \eqref{eq:uniprox} we get:
\begin{equation}
\sum_{\run=1}^{\nRuns}\step_{\run}^{2}\norm{\oper(\state_{\run+1/2})-\oper(\state_{\run})}^{2}_{\state_{\run+1/2},\ast}=\sum_{\run=1}^{\nRuns}\dfrac{\norm{\oper(\state_{\run+1/2})-\oper(\state_{\run})}^{2}_{\state_{\run+1/2},\ast}}{1+\sum_{j=1}^{\run-1}\norm{\oper(\state_{j+1/2})-\oper(\state_j)}^{2}_{\state_{j+1/2},\ast}}
\end{equation}
which by recalling \eqref{eq:logarithmic-residual}
we obtain:
\begin{equation}
\label{eq:squared}
\sum_{\run=1}^{\nRuns}\step_{\run}^{2}\norm{\oper(\state_{\run+1/2})-\oper(\state_{\run})}^{2}_{\state_{\run+1/2},\ast}=\bigoh(\log\nRuns)
\end{equation}
So, by combining \eqref{eq:first order} and \eqref{eq:squared} we readily obtain:
\begin{equation}
\text{$\dfrac{\sum_{\run=1}^{\nRuns}\step_{\run}^{2}\norm{\oper(\state_{\run+1/2})-\oper(\state_{\run})}^{2}_{\state_{\run+1/2},\ast}}{\sum_{\run=1}^{\nRuns}\step_{\run}}\to 0$ while $\nRuns \to +\infty$}
\end{equation}
\end{itemize}
Therefore, by letting $\nRuns \to +\infty$, the inequality \eqref{eq:last2} yields $\breg(\sol,\state_{\nRuns})\to -\infty$, contradiction. 

\para{Case 2: Under \eqref{eq:MS} condition}
Examining  the asymptotic behaviour of \eqref{eq:last2} term by term under the light of \eqref{eq:MS} condition we get the following:
\begin{itemize}
\item
For $\sum_{\run=1}^{\nRuns}\step_{\run}$, \eqref{eq:MS} guarantees by \eqref{eq:linear growth}:
\begin{equation}
\sum_{\run=1}^{\nRuns} \step_{\run}=\Omega(\nRuns)\;\;\text{and}\;\;\sum_{\run=1}^{\nRuns}\step_{\run}\to +\infty
\end{equation}
\item
For $\frac{\sum_{\run=1}^{\nRuns}\step_{\run}^{2}\norm{\oper(\state_{\run+1/2})-\oper(\state_{\run})}^{2}_{\state_{\run+1/2,\ast}}}{\sum_{\run=1}^{\nRuns}\step_{\run}}$, \eqref{lem:summability} guarantees:
\begin{equation}
\sum_{\run=1}^{\nRuns}\step_{\run}^{2}\norm{\oper(\state_{\run+1/2})-\oper(\state_{\run})}^{2}_{\state_{\run+1/2,\ast}}=\bigoh(1)
\end{equation}
which combined with \eqref{eq:linear growth} gives us:
\begin{equation}
\frac{\sum_{\run=1}^{\nRuns}\step_{\run}^{2}\norm{\oper(\state_{\run+1/2})-\oper(\state_{\run})}^{2}_{\state_{\run+1/2,\ast}}}{\sum_{\run=1}^{\nRuns}\step_{\run}} \to 0
\end{equation}
\end{itemize} 
Therefore,  y letting $\nRuns \to +\infty$, the inequality \eqref{eq:last2} yields that $\breg(\sol,\state_{\nRuns})\to -\infty$, a contradiction.
\end{proof}

Having all this at hand, we are finally in the position to prove the main result of this section; namely the convergence of the actual iterates of the method.  For that we will need an intermediate lemma that  shall allow us to pass from a convergent subsequence to global convergence (see also \cite{Com01}, \cite{Pol87}). 

\begin{lemma}
\label{lem:quasi-Fejer}
Let $\chi \in (0,1]$, $(\alpha_{\run})_{\run \in \N}$, $(\beta_{\run})_{\run \in \N}$ non-negative sequences and
$(\varepsilon_{\run})_{\run \in \N}\in l^{1}(\N)$ such that $\run=1,2,\dotsc$:
\begin{equation}
\alpha_{\run+1}\leq \chi \alpha_{\run}-\beta_{\run}+\varepsilon_{\run}
\end{equation}
Then, $\alpha_{\run}$ converges.
\end{lemma}

\begin{proof}
First, one shows that $\alpha_{\run \in \N}$ is a bounded sequence.  Indeed, one can derive directly that:
\begin{equation}
\alpha_{\run+1}\leq \chi^{\run+1}\alpha_0+\sum_{k=0}^{\run}\chi^{\run-k}\varepsilon_k
\end{equation}
Hence, $(\alpha_{\run})_{\run \in \N}$ lies in $[0,\alpha_0+\varepsilon]$, with $\varepsilon=\sum_{\run=0}^{+\infty}\varepsilon_{\run}$. Now, one is able to extract a convergent subsequence $(\alpha_{k_{\run}})_{\run \in \N}$, let say $\lim_{\run \to +\infty}\alpha_{k_{\run}}=\alpha \in [0,\alpha_0+\varepsilon]$ and fix $\delta >0$. Then, one can find some $\run_0$ such that $\alpha_{k_{\run_0}}-\alpha < \frac{\delta}{2}$ and $\sum_{m>\run_{k_{\run_0}}}\varepsilon_m<\frac{\delta}{2}$. That said, we have:
\begin{equation}
0\leq \alpha_{\run}\leq \alpha_{k_{\run_0}}+\sum_{m>\run_{k_{\run_0}}}\varepsilon_m<\frac{\delta}{2}+\alpha+\frac{\delta}{2}=\alpha+\delta
\end{equation}
Hence, $\limsup_{\run}\alpha_{\run}\leq \liminf_{\run}\alpha_{\run}+\delta$. Since, $\delta$ is chosen arbitrarily the result follows.
\end{proof}

\begin{proof}[Proof of \cref{thm:last-MP}]
Once more, we shall treat each regularity class individually.
\para{Case 1: Under \eqref{eq:MB} condition}
For the \eqref{eq:MB}, b y denoting $\lim_{\run \to +\infty}\step_{\run}=\step_{\infty}$ case we shall consider two cases for the asymptotic behaviour of the step-size $\step_{\run}$.
\begin{itemize}
\item
\emph{\bf{$\step_{\infty} >0$:}}
By recalling the definition of $\step_{\run}$:
\begin{equation}
\step_{\run}=\frac{1}{\sqrt{1+\sum_{j=1}^{\run-1} \norm{\oper(\state_{j+1/2})-\oper(\state_{j})}_{\state_{j+1/2}}^{2}}}
\end{equation}
whereas by rearranging we get:
\begin{equation}
\sum_{j=1}^{\run-1}\norm{\oper(\state_{j+1/1})-\oper(\state_j)}^{2}_{\state_{j+1/2}}= \frac{1}{\step_{\run}^{2}}-1
\end{equation}
and hence:
\begin{equation}
\sum_{\run=1}^{+\infty}\norm{\oper(\state_{\run+1/1})-\oper(\state_\run)}^{2}_{\state_{\run+1/2}}= \frac{1}{\step_{\infty}^{2}}-1<+\infty
\end{equation}
Therefore, by recalling \eqref{eq:energy}, we have for solution of \eqref{eq:VI}, $\sol \in \points$
\begin{equation}
\breg(\sol,\state_{\run+1})\leq \breg(\sol,\state_{\run})-\step_{\run}\braket{\oper(\state_{\run+1/2})}{\state_{\run+1/2}-\sol}+\step_{\run}^{2}\norm{\oper(\state_{\run+1/2})-\oper(\state_{\run})}_{\run+1/2,\ast}^{2}
\end{equation}
which enables us to directly apply \cref{lem:quasi-Fejer} for $\alpha_{\run}=\breg(\sol,\state_{\run})$, $\beta_{\run}=\step_{\run}\braket{\oper(\state_{\run+1/2})}{\state_{\run+1/2}-\sol}$ and $\varepsilon_{\run}=\step_{\run}^{2}\norm{\oper(\state_{\run+1/2})-\oper(\state_{\run})}_{\state_{\run+1/2},\ast}^{2}$.
\item
\emph{\bf{$\step_{\infty}=0$:}}
Fix an equilibrium $\sol \in \sols$ and consider the "Bregman zone":
\begin{equation}
 \breg_{\eps}= \{\point \in \points :\breg(\sol,\point)<\eps \}
  \end{equation}
  By the assumption for the regularizer $\hreg$, it follows that there exists some $\delta >0$ such that:
  \begin{equation}
  B_{\delta}=\{\point \in \points :\norm{\sol-\point}<\delta \}
  \end{equation}
  is contained in $\breg_{\eps}$. Hence, by regularity assumption for the \eqref{eq:gap},
  it follows that:
  \begin{equation}
  \text{$\braket{\oper(\point)}{\point-\sol}\geq c>0$ for some $c\equiv c(\eps)>0$ and for all $\point \notin \breg_{\eps}$,}
  \end{equation}
  in particular,  for all $\point \in \breg_{2\eps}\setminus \breg_{\eps}$.
  Assume now that $\sol$ is a limit point of $\state_{\run}$, \ie $\state_{\run}\in \breg_{2\eps}$ for infinitely many $\run \in \N$. 
  Now, by the prox-step, we get:
   and hence,
  \begin{equation}
  \step_{\run}\braket{\oper(\state_{\run})}{\state_{\run}-\sol}\leq\braket{\nabla \hreg(\state_{\run})-\nabla \hreg(\state_{\run+1/2})}{\state_{\run}-\sol}
  \end{equation}
  whereas by \cref{lem:3point} and after rearranging we get:
   \begin{align*}
  \breg(\sol,\state_{\run+1/2})&\leq \breg(\sol,\state_{\run})-\step_{\run} \braket{\oper(\state_{\run})}{\state_{\run}-\sol}+\breg(\state_{\run},\state_{\run+1/2})\\
  &\leq \breg(\sol,\state_{\run})-\step_{\run} \braket{\oper(\state_{\run})}{\state_{\run}-\sol}+\max \{\breg(\state_{\run},\state_{\run+1/2}),\breg(\state_{\run},\state_{\run+1/2})\}
    \end{align*}
   Therefore due to \cref{lem:last1} we obtain:
  \begin{equation}
  \breg(\sol,\state_{\run+1/2})\leq \breg(\sol,\state_{\run})-\step_{\run} \braket{\oper(\state_{\run})}{\state_{\run}-\sol}+\frac{2\gbound^{2}}{\hstr}\step_{\run}^{2}
    \end{equation} 
    We consider two cases:
    \begin{enumerate}
    \item
    $\state_{\run} \in \breg_{2\eps}\setminus \breg_{\eps}$: Then, $\braket{\oper(\state_{\run})}{\state_{\run}-\sol}\geq c>0$. So,
    \begin{equation}
    \breg(\sol,\state_{\run+1/2})\leq \breg(\sol,\state_{\run})-c\step_{\run}+\frac{2\gbound^{2}}{\hstr}\step_{\run}^{2}
    \end{equation}
   Now, provided that $\frac{2\gbound^{2}\step_{\run}^{2}}{\hstr} \leq c\step_{\run}$ or equivalently
   $\step_{\run}\leq \frac{c\hstr}{2\gbound^{2}}$.  
  we get: $\breg(\sol,\state_{\run+1/2})\leq 2\eps$.
  \item
  $\state_{\run}\in \breg_{\eps}$: Then, in this case we have:
  \begin{equation}
   \breg(\sol,\state_{\run+1/2})\leq \breg(\sol,\state_{\run})+\frac{2\gbound^{2}}{\hstr}\step_{\run}^{2}     \end{equation}
   Again, provided that $\frac{2\gbound^{2}}{\hstr}\step_{\run}^{2}\leq \eps$ or equivalently
   $\step_{\run}\leq \frac{\sqrt{2\eps \hstr}}{2\gbound}$ we get $\breg(\sol,\state_{\run+1/2})\leq 2\eps$
    \end{enumerate}
    Therefore, by summarizing the above we get that if $\step_{\run}\leq \min \{\frac{\sqrt{2\eps \hstr}}{2\gbound},\frac{c \hstr}{2\gbound^{2}}\}$, we have that $\state_{\run+1/2} \in \breg_{2\eps}$
    whenever $\state_{\run} \in \breg_{2\eps}$.
    Going further, due to \cref{prop:2prox} by setting $\base=\sol$, $\point_{1}=\state_{\run+1/2}$, $\point_{2}^{+}=\state_{\run+1}$, $\point=\state_{\run}$, $\dvec_{1}=-\step_{\run}\oper(\state_{\run+1/2})$ and $\dvec_{2}=-\step_{\run}\oper(\state_{\run+1/2})$ we get:
    \begin{multline}
    \breg(\sol,\state_{\run+1})\leq \breg(\sol,\state_{\run})-\step_{\run}\braket{\oper(\state_{\run+1/2})}{\state_{\run+1/2}-\sol}+\step_{\run}\braket{\oper(\state_{\run+1/2})-\oper(\state_{\run})}{\state_{\run+1}-\state_{\run+1/2}}\\
    -\breg(\state_{\run+1},\state_{\run+1/2})-\breg(\state_{\run+1/2},\state_{\run})
        \end{multline}
        whereas by applying Fenchel's inequality we obtain:
          \begin{multline}
    \breg(\sol,\state_{\run+1})\leq \breg(\sol,\state_{\run})-\step_{\run}\braket{\oper(\state_{\run+1/2})}{\state_{\run+1/2}-\sol}+\frac{\step_{\run}^{2}}{2\hstr}\norm{\oper(\state_{\run+1/2})-\oper(\state_{\run})}_{\state_{\run+1/2},\ast}^{2}\\
  +\frac{\hstr}{2}\norm{\state_{\run+1}-\state_{\run+1/2}}_{\state_{\run+1/2}}^{2} -\breg(\state_{\run+1},\state_{\run+1/2})-\breg(\state_{\run+1/2},\state_{\run})
        \end{multline} 
        Now, since $\frac{\hstr}{2}\norm{\state_{\run+1}-\state_{\run+1/2}}_{\state_{\run+1/2}}^{2} -\breg(\state_{\run+1},\state_{\run+1/2})\leq 0$ by \eqref{eq:strong-local} we get:
        \begin{equation}
          \breg(\sol,\state_{\run+1})\leq \breg(\sol,\state_{\run})-\step_{\run}\braket{\oper(\state_{\run+1/2})}{\state_{\run+1/2}-\sol} +\frac{\step_{\run}^{2}}{2\hstr}\norm{\oper(\state_{\run+1/2})-\oper(\state_{\run})}_{\state_{\run+1/2},\ast}^{2}  
          \end{equation}
        which, in turn, by \eqref{eq:bounded variation} the above yields:
    \begin{equation}
      \breg(\sol,\state_{\run+1})\leq \breg(\sol,\state_{\run})-\step_{\run}\braket{\oper(\state_{\run+1/2})}{\state_{\run+1/2}-\sol}+\frac{C^{2}}{2\hstr}\step_{\run}^{2}
    \end{equation}
    with $C=2\gbound+\beta \frac{4\gbound}{\hstr}$.
    Recall that $\state_{\run+1/2}\in \breg_{2\eps}$ by our previous claim. We now consider the following two cases:
    \begin{enumerate}
    \item
    $\state_{\run+1/2}\in \breg_{2\eps}\setminus \breg_{\eps}$: In this case: $\braket{\oper(\state_{\run+1/2})}{\state_{\run+1/2}-\sol}\geq c>0$, so,
    \begin{equation}
    \breg(\sol,\state_{\run+1})\leq \breg(\sol,\state_{\run})-c\step_{\run}+\frac{C^{2}}{2\hstr}\step_{\run}^{2}
   \end{equation}
   which holds provided that $\frac{C^{2}\step_{\run}^{2}}{2\hstr}\leq c\step_{\run}$ or equivalently $\step_{\run}\leq \frac{2c\hstr}{C^{2}}$,
    \item
    $\state_{\run+1/2}\in \breg_{\eps}$: First recall that:
    \begin{equation}
    \breg(\state_{\run+1/2},\state_{\run+1})+\breg(\state_{\run+1},\state_{\run+1/2})\leq \frac{2\step_{\run}^{2}}{\hstr}\norm{\oper(\state_{\run+1/2})-\oper(\state_{\run})}^{2}_{\state_{\run+1/2},\ast}
    \leq \frac{2\step_{\run}^{2}}{\hstr}C^{2}
    \end{equation}
    Therefore, we get that:
    \begin{equation}
    \label{eq:iterate-dist}
    \norm{\state_{\run+1}-\state_{\run+1/2}}^{2}\leq \frac{4\mu^{2}C^{2}}{\hstr^{2}}\step_{\run}^{2}
    \end{equation}
    Now, let us define the following:
    \begin{equation}
     \breg_{\eps}(\alpha)=\max \{\breg(\sol,\point):\dist(\point,\breg_{\eps}(\sol))<\alpha \}
     \end{equation}
     Clearly,  $\breg_{\eps}(\alpha)$ is continuous relative to $\alpha$ and $\lim_{\alpha \to 0^{+}} \breg_{\eps}(\alpha)=\eps$. Therefore, we have:
     \begin{equation}
     \text{$\breg_{\eps}(\alpha)\leq \eps$\;\;  for all $\alpha \leq \alpha^{\ast}$ with $\alpha^{\ast}$ sufficiently small.}
     \end{equation}
  Moreover, due to \eqref{eq:iterate-dist}, we conclude that $\breg(\sol,\state_{\run+1})\leq 2 \eps$,
  provided that
  $\step_{\run}\leq \frac{\alpha^{\ast}}{2\mu C}{\hstr}$.
    \end{enumerate}
We conclude that $\state_{\run+1}\in \nhd_{2\eps}$ provided that $\state_{\run}\in \breg_{2\eps}$ and  $\step_{\run}\leq
 \min\{ \frac{2c\hstr}{\gbound^{2}},\frac{\sqrt{2\eps \hstr}}{2\gbound}, \frac{\alpha^{\ast}}{2\mu C}{\hstr}\}$.  Since, $\step_{\run}\to 0$ and $\state_{\run}\in \breg_{2 \eps}$ infinitely often (due to \cref{prop:sub-seq}) we conclude that $\state_{\run}\in \breg_{2\eps}$ for all sufficiently large $\run$. With $\eps >0$ being arbitrary, the result follows.
 \end{itemize}
 \para{Case 2: Under \eqref{eq:MS} condition} By plugging in $\alpha_{\run}=\breg(\sol,\state_{\run})$, $\beta_{\run}=\step_{\run}\braket{\oper(\state_{\run+1/2}}{\state_{\run+1/2}-\sol}$ and $\varepsilon_{\run}=\step_{\run}^{2}\norm{\oper(\state_{\run+1/2})-\oper(\state_{\run})}^{2}_{\state_{\run+1/2},\ast}$ in \cref{lem:quasi-Fejer}
and combine it with \cref{lem:summability}, we get $\inf_{\sol \in \sols}\norm{\sol,\state_{\run}}$ converges. Thus, the result follows by applying \cref{prop:sub-seq}
   \end{proof}

\section{Properties of Numerical Sequences}
\label{app:sequences}

In this appendix, we provide the necessary inequality of numerical sequences.
This inequality is due to \citefull{BL19} and \citefull{LYC18} and will play an indispensable role for establishing the last iterate convergence and universality of our method. 
\smallskip

\begin{lemma}
\label{lem:logarithmic-1}
For all non-negative numbers $\alpha_{1},\dotsc \alpha_{\run}$, the following inequality holds:
\begin{equation}
\sum_{\run=1}^{\nRuns}\dfrac{\alpha_{\run}}{1+\sum_{i=1}^{\run}\alpha_{i}}\leq 1+\log(1+\sum_{\run=1}^{\nRuns}\alpha_{\run})
\end{equation}
\end{lemma}

\begin{proof}
The lemma will proved by induction. The induction base $\nRuns=1$ holds, since:
\begin{equation}
\frac{\alpha_{1}}{1+\alpha_{1}}\leq 1\leq 1+\log(1+\alpha_{1})
\end{equation}
Assume now that the lemma holds for $\nRuns-1$. Then, we are left to show that it also holds for $\nRuns$.
Indeed, by the induction hypothesis, we get:
\begin{equation}
\sum_{\run=1}^{\nRuns}\dfrac{\alpha_{\run}}{1+\sum_{i=1}^{\run}\alpha_{i}}\leq 1+\log(1+\sum_{\run=1}^{\nRuns-1}\alpha_{\run})+\frac{\alpha_{\nRuns}}{1+\sum_{\run=1}^{\nRuns}\alpha_{\run}}
\end{equation}
Thus, in order to complete the induction it suffices to show that:
\begin{equation}
1+\log(1+\sum_{\run=1}^{\nRuns-1}\alpha_{\run})+\frac{\alpha_{\nRuns}}{1+\sum_{\run=1}^{\nRuns}\alpha_{\run}}\leq 1+\log(1+\sum_{\run=1}^{\nRuns}\alpha_{\run})
\end{equation}
By denoting $\point=\alpha_{\nRuns}/(1+\sum_{\run=1}^{\nRuns-1}\alpha_{\run})$, the above equation is equivalent:
\begin{equation}
\log(\point+1)-\frac{\point}{1+\point}\geq 0
\end{equation}
which can be straighforwardly checked since $H(\point)=\log(\point+1)-\frac{\point}{1+\point}\geq 0$ for all $\point \geq 0$. Therefore, the result follows.
\end{proof}

\bibliographystyle{plainfull}
\bibliography{bibtex/IEEEabrv,bibtex/Bibliography-PM}

\end{document}